\documentclass[1 [leqno,10pt]{amsart} 
\author{Stefano Scrobogna}
\title{On the global well-posedness of a  class of 2D solutions for the Rosensweig system of ferrofluids}

\usepackage[a4paper]{geometry}
\usepackage{empheq}
\usepackage{times}
\usepackage[T1]{fontenc}
\usepackage{amssymb, amsmath,  bm}
\usepackage{amsfonts}
\usepackage[english]{babel}
\usepackage{amssymb,amsthm}
\usepackage{mathtools}
\DeclareMathAlphabet{\mathcal}{OMS}{cmsy}{m}{n}
\usepackage{hyperref}
\usepackage{cite}
\allowdisplaybreaks[1]
\usepackage[bbgreekl]{mathbbol}
\DeclareSymbolFontAlphabet{\mathbb}{AMSb}
\DeclareSymbolFontAlphabet{\mathbbl}{bbold}
\usepackage{xfrac}

\newcommand{\dx}{\textnormal{d}{x}}
\renewcommand{\d}{\textnormal{d}}
\renewcommand{\div}{\textnormal{div}}
\newcommand{\fine}{\hfill$\blacklozenge$}
\newcommand{\curl}{\textnormal{curl}}
\newcommand{\pare}[1]{\left( #1 \right)}

\newcommand{\norm}[1]{\left\| #1 \right\|}
\newcommand{\av}[1]{\left| #1 \right|}
\newcommand{\bra}[1]{\left[ #1 \right]}
\newcommand{\set}[1]{\left\{ #1 \right\}}
\newcommand{\Hud}{{H}^{\frac{1}{2}}\pare{\bR^2}}

\newcommand{\cP}{\mathcal{P}}
\newcommand{\pa}{\partial^\alpha}

\newcommand{\cC}{\mathcal{C}}
\newcommand{\cL}{\mathcal{L}}

\newcommand{\cQ}{\mathcal{Q}}
\newcommand{\cF}{\mathcal{F}}
\newcommand{\cJ}{\mathcal{J}}
\newcommand{\cD}{\mathcal{D}}
\newcommand{\cE}{\mathcal{E}}

\newcommand{\cR}{\mathcal{R}}
\newcommand{\bR}{\mathbb{R}}
\newcommand{\bN}{\mathbb{N}}
\newcommand{\cG}{\mathcal{G}}

\newcommand{\ps}[2]{\pare{ \left. #1 \ \right| \ #2 }}

\newcommand{\hra}{\hookrightarrow}

\newcommand{\loc}{\textnormal{loc}}
\newcommand{\NS}{Navier-Stokes }

\theoremstyle{theorem}
\newtheorem{theorem}{Theorem}[section]
\newtheorem*{theorem*}{Theorem}
\newtheorem{prop}[theorem]{Proposition}
\newtheorem{lemma}[theorem]{Lemma}

\theoremstyle{definition}
\newtheorem{definition}[theorem]{Definition}
\newtheorem{rem}[theorem]{Remark}
\newtheorem{hyp}{Hypothesis}

\numberwithin{equation}{section}

\begin{document}

\AtEndDocument{\bigskip{\footnotesize
  \textsc{BCAM - Basque Center for Applied Mathematics,Mazarredo, 14,  E48009 Bilbao, Basque Country -- Spain} \par
  \textit{E-mail address:}  \texttt{\href{mailto:sscrobogna@bcamath.org}{sscrobogna@bcamath.org}}}}
  
  \thanks{The research of the  author was partially supported by the Basque Government through the BERC 2014-2017 program and by the Spanish Ministry of Economy and Competitiveness MINECO: BCAM Severo Ochoa accreditation SEV-2013-0323.}

 \maketitle
 
 \begin{abstract}
 We study study a class of 2D solutions of a Bloch-Torrey regularization of the Rosensweig system in the whole space, which arise when the initial data and the external magnetic field are 2D.  We prove that such solutions are globally defined if the initial data is in $ H^k\pare{\bR^2}, k\geqslant 1 $. 
 \end{abstract}

 \section{Introduction}

 A ferrofluid  is a liquid which presents ferromagnetic properties, i.e. it becomes strongly magnetizable in presence of an external magnetic field. Such material do not exist naturally in the environment but it was created in 1963 by NASA \cite{stephen1965low} with a very specific goal: to be used as a fuel for rockets in an environment without gravity, whence the necessity to be pumped applying a magnetic field. \\
 
 The versatility of such material and its peculiar property of being controlled via a magnetic field made it suitable to be later used in a whole variety of applications: ferrofluids are for instance used in loudspeakers in order to cool the coil and damp the cone \cite{Miwa2003}, as seals in magnetic hard-drives \cite{raj1982ferrofluid}, in order to reduce friction \cite{Huang2011} or enhance heat transfer \cite{LAJVARDI20103508}, \cite{Sheikholeslami2015}. We refer the interested reader to \cite{Zahn2001}, the introduction of \cite{NST2016} and references therein for a survey of potential applications of ferrofluids. \\
 
 Ferrofluids are collidal\footnote{A mixture in which one substance of microscopically dispersed insoluble particles is suspended throughout another substance.} made of nanoscale ferromagnetic particles of a compound containing iron, suspended in a fluid. They are magnetically soft, which means that they do not retain magnetization once there is no external magnetic field acting on them. \\
 
 On a physical point of view ferrofluids (FF) are very different from magnetohydrodynamical (MHD) fluids: the former are magnetizable fluids with very low electrical conductivity while the latter are nonmagnetizables and electrically conducting. There are two generally accepted models describing the evolution of a FF which are known under the name of their developers, the Rosensweig model \cite{Rosensweig}, and the Shiliomis model \cite{shiliomis1975non}. The mathematical analysis of such systems is very recent, in \cite{AH_Shilomis_weak},  \cite{AH_Shilomis_strong}, \cite{AH_Rosensweing_strong} and \cite{AH_Rosensweing_weak} it is proved that both Shiliomis and Rosensweig model admit global weak and local strong solutions in bounded, smooth subdomains of $ \bR^3 $.  The same authors then considered as well thermal and electrical conductivity as well as steady-state solutions of various ferrofluids systems in  \cite{AH12-2}, \cite{AH12}, \cite{AH13}, \cite{AH14}, \cite{AH15}, \cite{AH16} and \cite{HHl16}. \\

In the present paper we will consider the following regularization of Bloch-Torrey type of the Rosensweig model for homogeneous micropolar fluids 

\begin{equation}\label{eq:Rosensweig}\tag{$ \mathcal{R} $}
\left\lbrace
\begin{aligned}
& \rho_0 \pare{\partial_t u +u\cdot \nabla u } -\pare{\eta + \zeta} \Delta u + \nabla p = \mu_0 M\cdot \nabla H  + 2\zeta\ \curl \ \Omega, \\
& \rho_0 \kappa \pare{\partial_t \Omega +u\cdot \nabla \Omega} - \eta ' \Delta \Omega - \lambda' \nabla \div \ \Omega = \mu_0 \ M\times H + 2\zeta \pare{\curl\ u -2\Omega}, \\
& \partial_t M +u\cdot \nabla M - \sigma \Delta M = \Omega \times M -\frac{1}{\tau} \pare{M-\chi_0 H}, \\
& \div\pare{H+ M}=F, \\
& \div \ u = \curl \ H =0, \\
& \left. \pare{u, \omega, M, H}\right|_{t=0}= \pare{u_0, \Omega_0, M_0, H_0 } ,
\end{aligned}
\right.
\end{equation}
where the parameters $ \rho_0, \eta, \zeta, \mu_0, \kappa, \eta', \lambda', \sigma, \tau $ and $ \chi_0 $ have a physical meaning and are considered to be fixed and positive. The unknown $ u $ represents the linear velocity wile $ \Omega $ represents the angular velocity, $ M $ is the magnetizing  field and $ H $ is the effective magnetizing field. The equation 
$$ \div\pare{H+M}=F $$
 will be often denoted as the \textit{magnetostatic equation}. 
The parameter $ \sigma > 0 $ comes in play  when the diffusion of the spin magnetic moment is not negligible, we refer the reader to \cite{GaspariBloch}, and indeed it has a regularizing effect since in such regime the system \eqref{eq:Rosensweig} is purely parabolic. \\

The constant $ \chi_0 $ is a dimensionless value called magnetic susceptibility, for oil-based fluids (see \cite{RZ02}) usually $ \chi_0\in \bra{0.3, 4.3} $ while for water-based fluids $ 0<\chi_0 \ll 1 $. The critical value $ \chi_0 =0 $ implies that the medium is not magnetizable and hence there is not external magnet force exerted on the fluid. \\

The equations \eqref{eq:Rosensweig} are derived under the following hypothesis (see \cite{Rosensweig2002})
\begin{itemize}
\item[$ \triangleright $] The ferromagnetic particles suspended in the carrier fluid are spherical, 

\item[$ \triangleright $] The ferromagnetic particles have the same size and mass, 

\item[$ \triangleright $] The density of the ferromagnetic particles in the colloidal is homogeneous, 

\item[$ \triangleright $] No agglomeration effects are considered.
\end{itemize}

 The equations \eqref{eq:Rosensweig} are considered in the whole three-dimensional space in $ \bR^3\times \bR_+ $, and we assume that
\begin{equation}\label{eq:hypF}\tag{H1}
F=F\pare{x_1, x_2, t}, 
\end{equation}
i.e. the external magnetic field is independent of the vertical variables. 
 In such setting we are going to consider special solutions of \eqref{eq:Rosensweig} of the form
\begin{equation}\label{eq:special_solutions}\tag{H2}
\begin{aligned}
u & = \pare{u_1\pare{x_1, x_2, t}, u_2\pare{x_1, x_2, t}, 0}, \\
\Omega & = \pare{ 0, 0, \omega \pare{x_1, x_2, t}}, \\
M & = \pare{M_1\pare{x_1, x_2, t}, M_2\pare{x_1, x_2, t}, 0}. 
\end{aligned}
\end{equation}

With the hypothesis \eqref{eq:hypF} and  \eqref{eq:special_solutions} the system \eqref{eq:Rosensweig} becomes:
\begin{equation}\label{eq:Rosensweig2D}\tag{$ \mathcal{R}_{\text{2D}} $}
\left\lbrace
\begin{aligned}
& \rho_0 \pare{\partial_t u +u\cdot \nabla u } -\pare{\eta + \zeta} \Delta u + \nabla p = \mu_0 M\cdot \nabla H + 2\zeta \pare{\begin{array}{c}
\partial_2 \omega\\ -\partial_1\omega
\end{array}}, \\
& \rho_0 \kappa \pare{\partial_t \omega + u\cdot \nabla \omega} - \eta' \Delta \omega = \mu_0 M\times H + 2\zeta \pare{\curl \ u -2\omega}, \\
& \partial_t M + u\cdot \nabla M - \sigma \Delta M = \pare{
\begin{array}{c}
-M_2\\ M_1
\end{array}
}\omega -\frac{1}{\tau}\pare{M-\chi_0 H}, \\
& \div\pare{H+ M}=F, \\
& \div \ u = \curl \ H =0, \\
& \left. \pare{u, \omega, M, H}\right|_{t=0}= \pare{u_0, \omega_0, M_0, H_0 }. 
\end{aligned}
\right.
\end{equation}

\begin{rem}
We underline the fact that the symbols $ \Delta, \nabla $ do \textit{not} represent the same operators in\eqref{eq:Rosensweig} and \eqref{eq:Rosensweig2D}; in the former they represent respectively the three-dimensional Laplacian and gradient while in the latter they represent the bi-dimensional Laplacian and gradient. Only thanks to the hypothesis \eqref{eq:hypF} and  \eqref{eq:special_solutions} we can perform the identification of \eqref{eq:Rosensweig} and \eqref{eq:Rosensweig2D}. In order to avoid confusion we explicitly define here the operators appearing in \eqref{eq:Rosensweig2D}, even though they are nothing but the standard three dimensional operators restricted onto the space of functions satisfying the hypothesis \eqref{eq:hypF} and  \eqref{eq:special_solutions}. From now on the symbols $ \Delta, \nabla $ represent respectively the operators
\begin{align*}
\Delta = \partial_1^2+\partial_2^2, &&
\nabla = \pare{\begin{array}{c}
\partial_1 \\ \partial_2
\end{array}},
\end{align*}
and the transport form is defined as
\begin{equation*}
u\cdot \nabla A = \sum_{i=1}^2 u_i\partial_i A. 
\end{equation*}
In the same spirit the vector product is identified as the bilinear form
\begin{equation*}
\pare{A, B}\in \bR^2\times \bR^2 \mapsto A\times B = -A_1B_2 + A_2 B_1 \in \bR,
\end{equation*}
and the curl operator is the following operator
\begin{equation*}
\curl \ u = -\partial_2 u_1 + \partial_1 u_2, 
\end{equation*}
while given any vector field $ v = \pare{v_1, v_2}^{\intercal} $ we define as $ v^\perp = \pare{-v^2, v^1}^\intercal $. \fine
\end{rem}

\section{Results and notation}

\subsection{Main result and organization of the paper}  The following statement codifies the main result presented in this paper which concerns the global well-posedness of solutions in the form \eqref{eq:special_solutions} for the system \eqref{eq:Rosensweig2D}:

\begin{theorem}\label{thm:main result}
Let $ u_0, \omega_0, M_0, H_0\in H^k\pare{\bR^2} $ for some $ k\geqslant 1 $ such that $ \div\ u_0=0 $ and $ \div \pare{M_0\pare{x}+H_0\pare{x}} = F\pare{x, 0} $ and let $ \cG_F = \Delta^{-1}\nabla F \in W^{1, \infty}_{\loc}\pare{\bR_+; H^{k+1}\pare{\bR^2}} $, $ F\in L^2_{\loc}  \pare{\bR_+; L^2\pare{\bR^2}} $. The system \eqref{eq:Rosensweig2D} admits a unique global strong solution in  $ \cC\pare{\bR_+ ;H^k\pare{\bR^2} } $ which enjoys the following additional regularity
\begin{align*}
u, \omega, M, H\in L^\infty \pare{\bra{0, T};H^k\pare{\bR^2} }, &&
\nabla u, \nabla \omega,\nabla M,\nabla H\in L^2\pare{\bra{0, T};H^k\pare{\bR^2} }, 
\end{align*} 
for each $ T>0 $. 
\end{theorem}

\begin{rem}
The choice of an initial data in $ H^k\pare{\bR^2}, \ k\geqslant 1 $ is due to technical reasons. We expect to be able to prove global propagation of any Sobolev regularity when the initial data belongs to $ L^2\pare{\bR^2} $, obtaining an analogous result of what is already known for the 2D incompressible \NS equations. Such result will be the subject of future investigation. \fine
\end{rem}

As it is often the case in the study of existence and regularity for complex fluids the main difficulty in the present paper is the global analysis of the perturbations induced by the many nonlinear interactions of \eqref{eq:Rosensweig2D}. At first hence we study the natural $ L^2\pare{\bR^2} $--energy decay for smooth solutions of \ref{eq:Rosensweig2D}. Adapting the techniques of \cite{AH_Rosensweing_weak} to the present setting and exploiting some cancellation properties it is hence possible to prove that smooth, decaying at infinity solutions of \eqref{eq:Rosensweig2D} propagate globally $ L^2\pare{\bR^2} $ regularity. Unfortunately, contrarily to the incompressible \NS equations, such result is not sufficient in order to construct global-in-time $ L^2 $ solutions by mean of compactness methods. In fact a standard way to construct global weak solutions is to prove that, given a sequence $ \pare{U_n}_n $, 
\begin{align*}
\pare{U_n}_n & \hspace{5mm}\text{ is bounded in } L^p\pare{\left[0, T\right); X_0}, \\
\pare{\partial_t U_n}_n & \hspace{5mm}\text{ is bounded in } L^p\pare{\left[0, T\right); X_1}, 
\end{align*}
for a $ p\in\pare{1, \infty} $, any $ T>0 $ and $X_0\hra X_1 $, hence if there exists some space $ X $ such that\footnote{In such notation $ Z\Subset Y $ means that $ Z $ is compactly embedded in $ Y $. }
\begin{equation*}
X_0\Subset X \hra X_1, 
\end{equation*}
it is possible to apply Aubin-Lions lemma \cite{Aubin63} in order to deduce that the sequence $ \pare{U_n}_n $ is compact in $ L^p\pare{\left[0, T\right); X} $. Whence, if such bounds are proved to be true, a passage to the limit as $ n\to\infty $ concludes the construction. In the case of Rosensweig system though, as it was already remarked in \cite{AH_Rosensweing_weak}, the Lorentz force $ \textsf{F}^m = \mu_0 \ M\cdot\nabla H $ is only $ L^1\pare{\left[0, T\right); H^{-\sfrac{1}{2}}\pare{\bR^2}} $, i.e. the Lorentz force has not sufficiently regularity in-time in order to apply such technique. A way to bypass such problem is hence to construct global weak solutions in $ \Hud $; if such global bounds can be attained we can mange hence to prove that $ \textsf{F}^m \in  L^2\pare{\left[0, T\right); L^2\pare{\bR^2}}$, whose regularity in time is  hi enough in order to deduce existence of global weak solutions by means of compactness methods. \\
\noindent Next we investigate if these weak solutions constructed are sufficiently regular  to deduce global propagation of any Sobolev regularity. The answer is indeed affirmative, and the proof of such result is performed via an iterative argument; given a $ k\in\bN\setminus\set{0} $ we suppose that the system \eqref{eq:Rosensweig2D} is globally well-posed in $ H^j\pare{\bR^2}, \ j\in\set{0, \ldots , k-1} $ and we prove that the global propagation holds true as well in $ H^k\pare{\bR^2} $. Indeed if $ k=0\Rightarrow H^k=L^2 $ the propagation is true thanks to the global $ L^2 $--estimates. The main tool in order to prove such inductive argument are the technical estimates performed in Lemma \ref{lem:higher_energy_est_transport_term} and \ref{lem:higher_energy_est_bilinear_term}. \\

The proof of Theorem \ref{thm:main result} is divided in two parts such as in others works describing  global regularity of two-dimensional complex fluids systems (we refer for instance to  \cite{PZ12}, \cite{LZZ08}, \cite{LLW10}, \cite{DeAnna17}); at first, using a Galerkin approximation scheme, it is possible to prove the existence of global weak solutions. Next, assuming the initial data fulfills the regularity requirements stated in Theorem \ref{thm:main result}, we prove that the system \ref{eq:Rosensweig2D} propagates globally Sobolev regularity of any order greater or equal than one.\\

The paper is structured as follows

\begin{itemize}

\item Section \ref{sec:preliminaries} is a brief introduction to more or less well-known technical results and notation which will be used all along the present work.

\item In Section \ref{sec:apriori_estimates} we perform some a priori estimates on sufficiently smooth and decaying at infinity solutions of \eqref{eq:Rosensweig2D} in the same spirit as in \cite{AH_Rosensweing_weak} (and as well \cite{AH_Shilomis_weak} and \cite{PZ12} for some different systems). In detail we prove that such regularized solutions conserve globally $ L^2 $ energy thanks to some cancellation properties first remarked in \cite{AH_Rosensweing_weak} in the framework of bounded and smooth domains of $ \bR^3 $ and here adapted to our framework. Next we prove in Lemma \ref{lem:H12_energy_estimates} that as long as the hypothesis for the conservation of the $ L^2 $ energy are satisfied, and if the initial data is more regular (namely $ \Hud $), then the global propagation of energy can be extended to the $ \Hud $ level as well. We focus to prove the global propagation of the $ \Hud $ energy since such step will be required in order to construct global weak solutions in $ H^{\frac{1}{2}-\varepsilon}\pare{\bR^2}, \ \varepsilon>0 $, providing hence global weak solutions with high regularity.  

\item  In Section \ref{sec:weak_solutions} we construct global weak solutions of \eqref{eq:Rosensweig2D} under the stronger hypothesis of an initial data in $ \Hud $. It is hence in this section that this higher regularity (compared to classical Leray solutions, cf. \cite{Leray} or \cite{monographrotating}) assumption on the initial data is explained. The Lorentz force $ \textsf{F}^m = \mu_0 \ M\cdot \nabla H $ can be bounded in the space $ L^1_{\loc}\pare{\bR_+; H^{-\frac{1}{2}}\pare{\bR^2}} $ only with the bounds provided by the global conservation of energy at a $ L^2 $ level only (i.e. with the results of Lemma \ref{lem:L2_energy_estimates}). Such time-regularity is hence not sufficient in order to apply standard compactness theorems in functional spaces (such as the one in \cite{Aubin63}), whence the requirement of an initial data in $ \Hud $, which is again non-restrictive since the goal is to construct global \textit{strong} solutions for \eqref{eq:Rosensweig2D}. 

\item Lastly in Section \ref{sec:HE} we prove that, considered an initial data in $ H^k $ and an external magnetic field $ F $ sufficiently regular, we can propagate globally-in-time such Sobolev regularity. Such result is not a completely trivial deduction as it is pointed out in Remark \ref{rem:propagation_HE}; again the Lorentz force $ \textsf{F}^m $ lacks  the commutation properties which are characteristics for transport terms with isochoric velocity fields, whence a more careful energy bound, whose key feature is an iterative proof relying on the technical Lemmas \ref{lem:higher_energy_est_transport_term} and \ref{lem:higher_energy_est_bilinear_term}, is required. 

\end{itemize}

\subsection{Preliminaries and notation} \label{sec:preliminaries}
From now on for any Lebesgue or Sobolev space whose domain  $\Omega \subseteq \bR^d $ is not explicitly defined  it will be implicitly considered to be $ \Omega = \bR^2 $.  \\

All along this paper, given  a $ v $ such that $ \hat{v}\in L^1_{\loc}\pare{\bR^2} $ we define the family of operators $ \pare{\Lambda^s}_{s \in \bR} $ as
\begin{equation*}
\Lambda^{s} v = \cF^{-1}\pare{\Big. \av{\xi}^{s}\hat{v}\pare{\xi}}. 
\end{equation*}
Using such family of operators we can hence define the \textit{nonhomogeneous fractional Sobolev} space $ H^s\pare{\bR^2} $ as the space of tempered distributions $ v $ such that $ \hat{v}\in L^2_{\loc}\pare{\bR^2} $ and such that
\begin{equation*}
\pare{1+\Lambda}^s v \in L^2\pare{\bR^2}. 
\end{equation*}

\noindent There exists as well an \textit{homogeneous} counterpart of the fractional Sobolev space which consists of all the tempered distributions  $ \hat{v}\in L^1_{\loc}\pare{\bR^2} $ such that $ \Lambda^s v \in L^2\pare{\bR^2} $. In order to avoid notational confusion between homogeneous and nonhomogeneous Sobolev spaces we denote the former as $ \Lambda^s L^2\pare{\bR^2} $. For a much deeper discussion on homogeneous Sobolev spaces and their properties we refer the reader to \cite[Section 1.3]{BCD} and references therein.  \\

Given a vector field $ v = \pare{v_1, \ldots, v_N} $ for any $ N\in\bN $ we denote as $ \nabla v $ the Jacobian matrix of $ v $ i.e.
\begin{equation*}
\nabla v =\pare{\Big. \partial_i v_j}_{\substack{i=1, 2 \\ j =1, \ldots , N}}. 
\end{equation*}

\noindent
It is of interest to notice that if we define 
\begin{equation*}
\norm{\nabla v}_{L^2}^2 = \sum_{i=1}^2 \sum_{j=1}^N \int \av{\partial_i v_j\pare{x}}^2\dx, 
\end{equation*}
there exists a $ K>0 $  such that for any $ v\in \Lambda L^2\pare{\bR^2} $
\begin{equation*}
\frac{1}{K}\norm{\Lambda v}_{L^2} \leqslant \norm{\nabla v}_{L^2} \leqslant K \norm{\Lambda v}_{L^2}, 
\end{equation*}
we shall use such property continuously in what follows and, by extension, we will identify for any $ k\in\bN $ the equivalent quantities
\begin{align*}
\norm{\pare{1+\Lambda}^k v}_{L^2} && \text{and} && \pare{\sum_{\av{\alpha}=0}^k \norm{\pa v}_{L^2}^2}^{\sfrac{1}{2}},
\end{align*} 
where $ \pa $ is a differential operator of the form $ \pa = \partial_1^{\alpha_1} \partial_2^{\alpha_2} $ where $ \alpha =\pare{ \alpha_1,\alpha_2} $. \\

A point of interest is to understand how, given a $ \sigma \in \bR $, the operator $ \Lambda^\sigma $ acts  on a product of tempered distributions. The following result, which belongs to the mathematical folklore (see \cite{gallagher_schochet}, \cite{Scrobo_Froude_FS} just to make an example),  gives a very simple criterion:

\begin{lemma}\label{lem:prod_rules_Sobolev_2D}
Let $ s, t $ real values such that $ s, t < 1 $ and $ s+t >0 $, and let $ \Lambda^s u_1\in L^2\pare{\bR^2}, \Lambda^t u_2 \in L^2\pare{\bR^2} $, then $ \Lambda^{s+t-1} \pare{u_1 u_2} \in  L^2\pare{\bR^2} $. 
\end{lemma}

A space of great interest in Section \ref{sec:apriori_estimates} will be the fractional $ \Lambda^{\frac{1}{2}}L^2\pare{\bR^2} $ space. In particular we will require the following interpolation inequality whose proof is a straightforward consequence of the continuous embedding of $ \Lambda^{\frac{1}{2}}L^2\pare{\bR^2} $ in $ L^4\pare{\bR^2} $ (see \cite[Chapter 1]{BCD}):
\begin{lemma}
Let $ u\in H^1\pare{\bR^2} $, the following chain of inequalities holds true
\begin{equation*}
\norm{u}_{L^4}\leqslant C_1 \norm{\Lambda^{\frac{1}{2}} u}_{L^2}\leqslant C_2 \norm{u}_{L^2}^{\sfrac{1}{2}} \norm{\nabla u}_{L^2}^{\sfrac{1}{2}}. 
\end{equation*}
\end{lemma}

Next we state the following interpolation inequality which will be useful 

\begin{lemma}\label{lem:interpolation_inequality}
Let $ v \in H^{\sfrac{3}{2}} \pare{\bR^2} $. Then $ v\in L^\infty\pare{\bR^2} $ and there exist a $ C > 0 $ such that
\begin{equation}\label{eq:interpolation_inequality}
\norm{v}_{L^\infty\pare{\bR^2}}\leqslant C \norm{\Lambda^{\frac{1}{2}} v}_{L^2 \pare{\bR^2}}^{\sfrac{1}{2}} \norm{\Lambda^{\frac{3}{2}} v}_{L^2 \pare{\bR^2}}^{\sfrac{1}{2}}. 
\end{equation}
\end{lemma}

Indeed the fact that $ v\in H^{\frac{3}{2}} $ obviously implies that $ v\in L^\infty $ by classical Sobolev embeddings.  What is important in Lemma \ref{lem:interpolation_inequality} is the inequality \eqref{eq:interpolation_inequality} which allows qualitatively better control on the $ L^\infty $ norm of $ v $ in terms of an interpolation between high and low order derivatives. \\

We provide a short proof of the classical result stated in Lemma \ref{lem:interpolation_inequality} for the sake of clarity. 

\begin{proof}
We can indeed decompose $ v $ as $ v=v_A + v^A $, where 
\begin{align*}
v_A = \cF^{-1}\pare{1_{\set{\av{\xi}\leqslant A}} \hat{v}}, &&
v^A = \cF^{-1}\pare{1_{\set{\av{\xi}> A}}\hat{v}}.
\end{align*}
Using a Bernstein inequality (see \cite[Lemma 2.1, p. 52]{BCD}) and the Sobolev embedding $ \Lambda^{\sfrac{1}{2}}L^2\hra L^4 $ we can argue that
\begin{equation*}
\norm{v_A}_{L^\infty}\lesssim A^{\sfrac{1}{2}} \norm{v}_{L^4} \leqslant A^{\sfrac{1}{2}} \norm{ \Lambda^{\sfrac{1}{2}} v}_{L^2}, 
\end{equation*}
while for the hi-frequency part $ v^A $ we can argue in the following way
\begin{align*}
\av{v^A\pare{x}} & \lesssim \int _{\set{\av{\xi}>A}}\av{\hat{v}\pare{\xi}}\d \xi, \\
& \lesssim \pare{\int _{\set{\av{\xi}>A}} \av{\xi}^{-3}\d \xi}^{\sfrac{1}{2}} \norm{ \Lambda^{\sfrac{3}{2}} v}_{L^2}
\lesssim A^{-\sfrac{1}{2}}\norm{ \Lambda^{\sfrac{3}{2}} v}_{L^2}. 
\end{align*}
Setting hence
\begin{equation*}
A = \frac{\norm{ \Lambda^{\sfrac{3}{2}} v}_{L^2}}{\norm{ \Lambda^{\sfrac{1}{2}} v}_{L^2}}, 
\end{equation*}
we conclude. 
\end{proof}

And we state the following simple product rule, for a proof of which we refer the reader to \cite[Corollary 2.54, p. 90]{BCD}

\begin{lemma}\label{lem:product rule}
For each $ s>0 $ the space $ H^s\cap L^\infty $ is an algebra and there exist a $ C>0 $ such that for each $ v_1, v_2 \in H^s\cap L^\infty $
\begin{equation*}
\norm{v_1 v_2}_{H^s}\leqslant \frac{C^{s+1}}{s} \pare{\Big. \norm{v_1}_{L^\infty} \norm{v_2}_{H^s} + \norm{v_1}_{H^s} \norm{v_2}_{L^\infty} }. 
\end{equation*}
A similar results holds for the homogeneous space $ \Lambda^s L^2\cap L^\infty $. 
\end{lemma}

\noindent
We will use these results repeatedly in the following. \\

\begin{definition}[Hodge decomposition]
Let $ v \in L^2 $, we can decompose $ v $ as $ v = \cP v + \cQ v $, where
\begin{equation}
\div \ \cP v =0, \hspace{1cm} \curl \ \cQ v =0, 
\end{equation}
and 
\begin{align*}
\cP v = \pare{\Big. 1_{\bR^3} - \mathcal{R}\otimes \mathcal{R}} v, && \cQ v = \mathcal{R}\otimes \mathcal{R} \ v, 
\end{align*}
where $ \cR=\pare{\cR_1, \cR_2} $ and $ \mathcal{R}_j $ is the $ j $-th Riesz transform (cf. \cite{Stein70} and \cite{Stein93}). 
\end{definition}

\begin{rem}
We will denote the Leray projector $ \cP $ and its orthogonal (in $ L^2 $) complement $ \cQ $ respectively as
\begin{equation*}
\cP = 1- \Delta^{-1}\nabla\div, \hspace{1cm} \cQ = \Delta^{-1}\nabla\div, 
\end{equation*}
following a common notation in the Navier-Stokes theory (see \cite{LR1}). The operators $ \cP $ and $ \cQ $ are continuous in $ L^p, \ p \in\pare{1, \infty} $. \fine
\end{rem}

\begin{rem}
Let us consider a Banach space $ X $ and let us set $ p\in\bra{1, \infty} $, we  say that
\begin{equation*}
\text{the sequence } \pare{u_n}_{n\in\bN} \text{ is uniformly bounded in } L^p_{\loc}\pare{\bR_+; X}, 
\end{equation*}
if, fixed any $ T > 0 $, there exists a positive constant $ c_T $ depending on $ T>0 $ only such that
\begin{equation*}
\norm{u_n}_{L^p\pare{\bra{0, T} ; X}}\leqslant c_T \text{ for any } n\in \bN. 
\end{equation*}
\fine
\end{rem}

We will denote as $ C $ a positive constant whose  expression may depend upon the several physical parameters appearing in \eqref{eq:Rosensweig2D} and whose explicit value may implicitly vary from line to line.

\section{Energy inequality and a priori estimates}\label{sec:apriori_estimates}

In this section we perform some a priori estimates for smooth solutions of \eqref{eq:Rosensweig2D} which decay at infinity sufficiently fast to zero so that we can integrate by parts without boundary terms.  We moreover consider the external magnetic field $ F $ to be of zero average, and we formally define 
\begin{equation*}
\cG_F = \Delta^{-1}\nabla F. 
\end{equation*}\\

The main result we want to prove in this section is the following one 
\begin{prop}\label{prop:energy_est}
Let $ u_0, \omega_0, M_0, H_0\in H^{\frac{1}{2}}\pare{\bR^2} $ and let us suppose $ u, \omega, M, H $ is a smooth solution of \eqref{eq:Rosensweig2D} which decays at infinity sufficiently fast so that there is no boundary term. Then if $ F\in L^2_{\loc} \pare{\bR_+; L^2\pare{\bR^2}} $ and $ \cG_F=\Delta^{-1}\nabla F\in W^{1, \infty}_{\loc} \pare{\bR_+; H^{\frac{3}{2}}\pare{\bR^2}} $ then for any $ T>0 $
\begin{align*}
\left( u, \omega, M, H \right) \in L^\infty\pare{\bra{0, T} ; \Hud }, && \nabla\left( u, \omega, M, H \right) \in L^2\pare{\bra{0, T} ; \Hud }.
\end{align*}
\end{prop}

The proof of Proposition \ref{prop:energy_est} is divided in two steps; at first we prove that smooth, decaying at infinity solutions of \eqref{eq:Rosensweig2D} conserve $ L^2 $ energy adapting the proof of \cite{AH_Rosensweing_weak} in our case, i.e. when the domain is $ \bR^2 $, next we use the propagation of the $ L^2 $ energy in order to prove that, when the initial data belongs to $ H^{\sfrac{1}{2}} $, such smooth and decating solutions propagate the $ H^{\sfrac{1}{2}} $ regularity as well. The complete proofs of such results are performed in full detail in Appendix \ref{appendix_technical_estimates}, but we will nonetheless explain the main feature and cancellations which make possible such propagation. \\

 Let us define the following quantities:

\begin{align}
\tilde{c} & = \min\set{ \eta, \eta', \sigma, \frac{\mu_0\sigma}{2}, \frac{1}{ \tau} \pare{ \frac{\mu_0}{2} + \chi_0\pare{\mu_0+\frac{1}{2}}}, \frac{1}{\tau} },\label{eq:ctilde} \\
\cE\pare{0} & =  {\rho_0}\norm{u_0}^2_{L^2} + {\mu_0} \   \norm{H_0}^2_{L^2} + \rho_0 \kappa  \norm{\omega_0}_{L^2}^2 + \norm{M_0}_{L^2}^2,  \\
\cE\pare{t} & = {\rho_0}\norm{u}^2_{L^2} + {\mu_0} \   \norm{H}^2_{L^2} + \rho_0 \kappa  \norm{\omega}_{L^2}^2 + \norm{M}_{L^2}^2,  \\
\cE_d \pare{t} & =  \norm{\nabla u}^2_{L^2} + \norm{\nabla\omega}^2_{L^2} +  \norm{\nabla M}_{L^2}^2 +  \norm{\div\ M}_{L^2}^2 +\norm{H}_{L^2}^2 +\norm{M}_{L^2}^2, \label{eq:def_Ed}\\
f_\tau \pare{t} & = \frac{C}{\tau}\norm{\cG_F}_{L^2}^2 + C \norm{F}_{L^2}^2,\label{eq:cFtau} 
\end{align}
while we will denote as $ C $ a positive constant, whose explicit value may vary from line to line, which depends upon the physiscal quantities  $ \rho_0, \eta, \zeta, \mu_0, \kappa, \eta' , \sigma $ and $ \tau $.

The first result is the following control of $ L^2 $ energy:
\begin{lemma} \label{lem:L2_energy_estimates}
Let $ u_0, \omega_0,  M_0, H_0 \in L^2 $ and $F,  \cG_F, \partial_t \cG_F \in L^2_{\loc}\pare{\bR_+; L^2}  $. 
Let $ u, \omega, M, H $ be a smooth, decaying at infinity, solution of \eqref{eq:Rosensweig2D} in the timespan $ \bra{0, T} $ where $ T $ is positive and possibly finite. Then
\begin{align*}
u, M, H, \omega \in L^\infty\pare{\bra{0, T}; L^2}, && \nabla u, \nabla \omega, \nabla M, H, M \in L^2\pare{\bra{0, T}; L^2},
\end{align*}
and for each $ t\in\bra{0, T} $ the following  inequality holds
\begin{equation}\label{eq:L2_energy_bound}
\frac{1}{2} \ \cE\pare{t} + \tilde{c}\int_0^t \cE_d  \pare{t'}\d t' \leqslant \Psi\pare{U_0, F, \cG_F} ,  
\end{equation}
where
\begin{equation}\label{eq:def_Psi}
\Psi\pare{U_0, F, \cG_F}= 
\frac{1}{2} \ \cE\pare{0} + C \pare{\norm{F}_{L^2\pare{\bra{0, T}; L^2}} + \norm{\cG_F}_{L^2\pare{\bra{0, T}; L^2}} + \norm{\partial_t\cG_F}_{L^2\pare{\bra{0, T}; L^2}}}. 
\end{equation}
\end{lemma}

It is a rather common feature, when it comes to estimate the $ L^2 $ regularity of smooth solution of nonlinear parabolic system, to look for suitable cancellations in the energy terms of higher order. This procedure leads to the well-known cancellation
\begin{equation*}
\ps{u \cdot \nabla u}{u}_{L^2} = \frac{1}{2} \int u \cdot \nabla \pare{\av{u}^2} \ \dx =0,
\end{equation*} 
due to the incompressibility of the flow, which makes possible to prove the existence of global  solutions \textit{\`a la Leray} for the incompressible \NS\ equations in any dimension. We obviously refer to seminal work of Leray \cite{Leray}.  Such simple trick can hence be used in relatively simple systems describing hydrodynamical incompressible inhomogeneous flows, such as in  \cite{charve2} or \cite{Scrobo_Froude_periodic}, in order to deduce immediately the existence of global $ L^2 $ energy solutions. \\

More refined cancellations   can be used as well in more sophisticated systems describing complex fluids, such as the already mentioned \cite{PZ12}, \cite{AH_Shilomis_weak} or \cite{AH_Rosensweing_weak}, in order to obtain uniform $ L^2 $ bounds for regularized solutions. We outline hence here the main cancellations required in order to achieve such uniform bounds, leaving the detailed computations to Appendix \ref{pf:L2_bounds}, in order to, hopefully, provide a clear idea of the methodology adopted without the unnecessary burden of the long, albeit inevitable, estimates which are involved. We remark that the cancellations explained in the following have been at first performed in \cite{AH_Rosensweing_weak} in the context of bounded, smooth, three-dimensional domains.  \\

Indeed the transport terms $ u\cdot \nabla u, \ u\cdot \nabla \omega, \ u\cdot \nabla M $ give a zero contribution in a $ L^2 $ estimate due to the fact that we consider the velocity flow to satisfy the incompressibility condition $ \div \ u=0 $, i.e.
\begin{align*}
\ps{ u\cdot \nabla u}{u}_{L^2}=0, && \ps{ u\cdot \nabla \omega}{\omega}_{L^2} =0, && \ps{ u\cdot \nabla M}{M}_{L^2} =0, 
\end{align*}
exploiting the very same trick mentioned above. Moreover since the vector $ \pare{-M_2, M_1} $ is orthogonal to $ M $ the pointwise identity $ \pare{-M_2, M_1}\cdot M=0 $ holds true, whence
\begin{equation*}
\ps{\pare{
\begin{array}{c}
-M_2\\ M_1
\end{array}
}\omega}{M}_{L^2}=0. 
\end{equation*}

The only bilinear interactions which do not present an immediate cancellation are hence the \textit{Lorentz force} $ \textsf{F}^m = \mu_0 M\cdot \nabla H $ appearing in the equation describing the evolution of $ u $ in \eqref{eq:Rosensweig2D}, and the bilinear interaction $ \mu_0 M\times H $ in the equation for $ \omega $. The key observation is hence the following one: since $ \curl \ H = 0 $ the vector field $ H $ can hence be written as the gradient of a potential function $ \phi_H $, with such the following chain of identities holds true\footnote{Here we use Einsteins summation convention}
\begin{equation*}
\ps{M\cdot \nabla H }{u}_{L^2} = \int M_i \partial_i \partial_j \phi_H \ u_j \overset{\div \ u=0}{=} - \int \partial_j M_i \partial_i \phi_H \ u_j = -\ps{u\cdot\nabla M}{H}_{L^2}. 
\end{equation*}
We can recover an analogous term as the r.h.s. of the above equation multiplying the equation of $ M $ for $ H $ and integrating, in doing so we provide the cancellation required for the Lorentz force $ \textsf{F}^m $, but we create an additional bilinear interaction 
\begin{equation*}
\ps{\pare{
\begin{array}{c}
-M_2\\ M_1
\end{array}
}\omega}{H}_{L^2}=0. 
\end{equation*}
Fortunately such interaction cancels with the term\footnote{One convinces himself performing the computations componentwise} $ \ps{M\times H}{\omega}_{L^2} $, which indeed is exactly the second term of which we could not identify an immediate cancellation. \\

Next we state and prove the following simple lemma, which relates the regularity of the vector field $ H $ in terms of the regularity of $ M $. Despite the proof is immediate we will use continuously such technical result in the rest of the paper.

\begin{lemma}\label{lem:reg_H}
Let us fix a $ s\in\bR $ and let $ M, \cG_F $ be such that $ {\Lambda^s M, \Lambda^s \cG_F\in L^2} $, then there exists a positive constant $ C $ such that
\begin{equation*}
\begin{aligned}
\norm{ \Lambda^s H}_{L^2} & \leqslant C \pare{\norm{\Lambda^s M}_{L^2} + \norm{\Lambda^s \cG_F}_{L^2}}, \\
\norm{ \Lambda^s \nabla H}_{L^2} & \leqslant C \pare{\norm{\Lambda^s \nabla M}_{L^2} + \norm{\Lambda^s \nabla \cG_F}_{L^2}}
.
\end{aligned}
\end{equation*}
\end{lemma}

\begin{proof}
Let us consider the magnetostatic equation $ \div\pare{M+H}=F $,  we can deduce hence that
\begin{equation}\label{eq:H_in_fz_di_M}
H = -\cQ M +\cG_F, 
\end{equation}
where $ \cQ $ is defined as
\begin{equation*}
\cQ v = \Delta^{-1}\nabla\div  \ v, 
\end{equation*}
and for any vector field $ v $ we have $ \curl \ \cQ v=0 $. The operator  $ \cQ $ commutes with $ \Lambda^s $, and being $ \cQ $ a Fourier multiplier of order zero it maps  $ L^2 $  to itself, whence the claim follows. 
\end{proof}

A first application of Lemma \ref{lem:reg_H} is the following lemma, which provides a bound for the Lorentz force $ \textsf{F}^m $ in $ H^{\sfrac{1}{2}} $:

\begin{lemma}\label{lem:bound_Lorentz_force}
The following bound holds true
\begin{equation*}
\av{ \ps{\Lambda^{\sfrac{1}{2}} \pare{M\cdot \nabla H}}{\Lambda^{\sfrac{1}{2}} u}_{L^2}}\leqslant C \norm{\nabla u}_{L^2} \norm{ \Lambda^{\sfrac{1}{2}} M}_{L^2} \pare{\norm{\nabla \Lambda^{\sfrac{1}{2}} M}_{L^2} + \norm{\nabla \Lambda^{\sfrac{1}{2}} \cG_F}_{L^2}}. 
\end{equation*}
\end{lemma}

\begin{proof}
Since for any $ s\in \bR $ the operator $ \Lambda^s v = \mathcal{F}^{-1} \pare{ \av{\xi}^s\hat{v}} $ is self-adjoint in $ L^2 $ we deduce that
\begin{align*}
\av{ \ps{ \Lambda^{\sfrac{1}{2}}\pare{M\cdot \nabla H}}{\Lambda^{\sfrac{1}{2}} u}_{L^2}} & =
\av{\int \Lambda^{\sfrac{1}{2}} \pare{M\cdot\nabla H} \cdot \Lambda^{\sfrac{1}{2}} u\ \dx},  \\
& =  \av{ \int \pare{M\cdot\nabla H}  \cdot \Lambda u \ \dx }, \\
& \leqslant \norm{\nabla u}_{L^2} \norm{M\cdot \nabla H}_{L^2}. 
\end{align*}
whence applying Lemma \ref{lem:prod_rules_Sobolev_2D} and \ref{lem:reg_H}
\begin{align*}
\norm{M\cdot \nabla H}_{L^2} & \leqslant C \norm{\Lambda^{\sfrac{1}{2}} M}_{L^2} \norm{\Lambda^{\sfrac{1}{2}}\nabla H}_{L^2}, \\
& \leqslant C  \norm{\Lambda^{\sfrac{1}{2}} M}_{L^2} \pare{\norm{\Lambda^{\sfrac{1}{2}}\nabla M}_{L^2} + \norm{\Lambda^{\sfrac{1}{2}}\nabla \cG_F}_{L^2}}.
\end{align*}
\end{proof}

Now we can pass to the second step in the proof of Proposition \ref{prop:energy_est}, i.e. we provide global $ H^{\frac{1}{2}}\pare{\bR^2} $ bounds for smooth, decaying at infinity, solutions of \eqref{eq:Rosensweig2D}. Let us hence define the following quantities, 
\begin{equation}
\label{eq:def_c}
\begin{aligned}
c & = \min \set{  \frac{\eta+\zeta}{2}, \frac{\eta'}{2} , 4\zeta , \frac{\sigma}{2} , \frac{1}{\tau} , \frac{\chi_0}{2\tau} }, \\
\cF\pare{0} & = \rho_0 \norm{\Lambda^{\sfrac{1}{2}} u_0}_{L^2}^2 + \rho_0\kappa \norm{ \Lambda^{\sfrac{1}{2}} \omega_0}_{L^2}^2 + \norm{\Lambda^{\sfrac{1}{2}} M_0}_{L^2}^2, \\
\cF\pare{t} & = \rho_0 \norm{\Lambda^{\sfrac{1}{2}} u}_{L^2}^2 + \rho_0\kappa \norm{\Lambda^{\sfrac{1}{2}} \omega}_{L^2}^2 + \norm{\Lambda^{\sfrac{1}{2}} M}_{L^2}^2, \\
\cF_d\pare{t} & = \norm{\Lambda^{\sfrac{1}{2}}\nabla u}_{L^2} + \norm{\Lambda^{\sfrac{1}{2}}\nabla \omega}_{L^2}^2 + \norm{\Lambda^{\sfrac{1}{2}} \omega}_{L^2}^2 + \norm{\Lambda^{\sfrac{1}{2}} \nabla M}_{L^2}^2 + \norm{\Lambda^{\sfrac{1}{2}} M}_{L^2}^2 +\norm{\Lambda^{\sfrac{1}{2}}\cQ M}_{L^2}^2, \\
\Phi_\tau \pare{t} & = \norm{\Lambda^{\sfrac{1}{2}}\nabla \cG_F}_{L^2}^2 + \frac{C}{\tau}\norm{\Lambda^{\sfrac{1}{2}} \cG_F}_{L^2}^2+
\norm{\nabla \omega}_{L^2}\norm{\nabla u}_{L^2}. 
\end{aligned}
\end{equation}
which will be used in the statement of the following lemma:

\begin{lemma}
\label{lem:H12_energy_estimates}
Let $ u_0, \omega_0,  M_0, H_0 \in H^{\sfrac{1}{2}} $,  $F\in L^2_{\loc}\pare{\bR_+; L^2}  $ and $\cG_F\in  W^{1, \infty}_{\loc}\pare{\bR_+; H^{\sfrac{3}{2}}} $. 
Let $ u, \omega, M, H $ be a smooth, decaying at infinity, solution of \eqref{eq:Rosensweig2D} in the timespan $ \bra{0, T} $ where $ T $ is positive and possibly finite. Then 
\begin{equation*}
\Lambda^{\sfrac{1}{2}}\pare{u, \omega, M} \in L^\infty\pare{\bra{0, T}; L^2 }, \hspace{1cm} \Lambda^{\sfrac{1}{2}}\pare{\nabla u, \nabla \omega,  \nabla M} \in L^\infty\pare{\bra{0, T}; L^2 },
\end{equation*}
and for each $ t\in\bra{0, T} $ the following inequality holds
\begin{equation*}
\cF \pare{t} + 2c \int_0^t  \cF_d \pare{t'} \d t' \leqslant \widetilde{\Psi}\pare{U_0, F, \cG_F}, 
\end{equation*}
where $ \widetilde{\Psi}\pare{U_0, F, \cG_F} $ is defined as
\begin{multline*}
\widetilde{\Psi}\pare{U_0, F, \cG_F} =
C \  \cF\pare{0}\exp\set{\frac{C}{\tilde{c}} \ \Psi\pare{U_0, F, \cG_F}} + C \ \exp\set{\frac{C}{\tilde{c}} \ \Psi\pare{U_0, F, \cG_F}} \norm{\cG_F}_{L^2\pare{\bra{0, T}; H^{\frac{3}{2}}}} \\
 + \frac{C}{\tilde{c}} \ \exp\set{\frac{C}{\tilde{c}} \ \Psi\pare{U_0, F, \cG_F}} \ \Psi\pare{U_0, F, \cG_F},
\end{multline*}
and $ \Psi $ is defined in \eqref{eq:def_Psi}. Moreover
\begin{equation}\label{eq:H_in_Eud}
\Lambda^{\sfrac{1}{2}} H \in L^\infty\pare{\bra{0, T}; L^2}, \hspace{1cm} \Lambda^{\sfrac{1}{2}}\nabla H\in  L^2\pare{\bra{0, T}; L^2}. 
\end{equation}
\end{lemma}

The proof of Lemma \eqref{lem:H12_energy_estimates} consists simply in performing $ \Lambda^{\sfrac{1}{2}} L^2 $ energy estimates on the system \ref{eq:Rosensweig2D} and using the regularity results stated in Lemma \ref{lem:L2_energy_estimates} and proved in Appendix \ref{pf:L2_bounds} to deduce a global energy bound. The detailed proof is postponed in Appendix \ref{appendix:H12_energy_estimates}, and can be skipped in a first stance, for the sake of the readability. Nonetheless one convinces himself that such estimates work out fine considering the bound provided in Lemma \ref{lem:bound_Lorentz_force}: if we consider $ \cG_F $ sufficiently regular\footnote{Let us recall that $ \cG_F = \Delta^{-1}\nabla F $ depends on the external magnetic field $ F $ only, whence it is not an unknown of the system.} the bound provided allows us to apply a Gronwall inequality and entails global $ \Lambda^{\sfrac{1}{2}} L^2 $ regularity provided that $ \norm{\nabla u}_{L^2} \in L^2_{\loc}\pare{\bR_+} $, which is assured by Lemma \ref{lem:L2_energy_estimates}. \\

At this point we can hence prove Proposition \ref{prop:energy_est}; denoting $ U=\pare{u, \omega, M, H} $, considering the inequality
\begin{align*}
\norm{U}_{\Hud} & = \sqrt{ \norm{U}_{L^2}^2 + \norm{\Lambda^{\sfrac{1}{2}} U}_{L^2}^2}, \\
& \leqslant \norm{U}_{L^2} + \norm{\Lambda^{\sfrac{1}{2}} U}_{L^2}, 
\end{align*}
and the results of Lemma \ref{lem:L2_energy_estimates} and \ref{lem:H12_energy_estimates} the claim of Proposition \ref{prop:energy_est} follows.

\section{The approximate system}\label{sec:weak_solutions}

The purpose of the present section is to build global weak solutions for the system \eqref{eq:Rosensweig2D} for sufficiently regular initial data. We will use a Galerkin approximation method.   In detail we prove the following result:

\begin{prop}\label{prop:existence_global_weak_H12_solutions}
Let $ u_0, \omega_0, M_0, H_0 \in H^{\frac{1}{2}}\pare{\bR^2} $ be such that $ \div \ u_0=0 $. Let $F\in L^{2}_{\loc}\pare{\bR_+; L^2}, \ \cG_F \in W^{1, \infty}_{\loc} \pare{\bR_+; H^{\sfrac{3}{2}}}  $, then there exists a unique global weak solution $ \pare{u, \omega, M, H} $ of \eqref{eq:Rosensweig2D} in the energy space 
\begin{equation*}
\pare{u, \omega, M, H}\in \cC\pare{\bR_+; H^{\frac{1}{2}}}, \hspace{1cm} \pare{ \nabla u, \nabla \omega, \nabla M, \nabla H}\in L^2_{\loc}\pare{\bR_+; H^{\frac{1}{2}}}. 
\end{equation*}
\end{prop}

\begin{proof}

Let us define the following truncation operator
\begin{equation*}
\cJ_n v = \cF^{-1}\pare{1_{\set{\frac{1}{n}\leqslant\av{\xi}\leqslant n}} \hat{v}\pare{\xi}}, 
\end{equation*}
which localize a tempered distribution $ v $ away from  low and  high frequencies. With such we can define the following sequence of approximating systems of \eqref{eq:Rosensweig2D}: 

\begin{equation}\label{eq:Rosensweig2D_approximated1}
\left\lbrace
\begin{aligned}
& \rho_0 \pare{\partial_t u_n +\cJ_n \pare{u_n \cdot \nabla u_n} } -\pare{\eta + \zeta} \Delta u_n + \nabla p_n = \mu_0 \  \cJ_n \pare{  M_n\cdot \nabla H_n} + 2\zeta \pare{\begin{array}{c}
\partial_2 \omega_n\\ -\partial_1\omega_n
\end{array}}, \\
& \rho_0 \kappa \pare{\partial_t \omega_n + \cJ_n\pare{ u_n\cdot \nabla \omega_n}} - \eta' \Delta \omega_n = \mu_0 \cJ_n \pare{ M_n\times H_n} + 2\zeta \pare{\curl \ u_n -2\omega_n}, \\
& \partial_t M_n + \cJ_n\pare{ u_n\cdot \nabla M_n} - \sigma \Delta M_n = \cJ_n\pare{\pare{
\begin{array}{c}
-M_{2, n}\\ M_{1, n}
\end{array}
}\omega_n} -\frac{1}{\tau}\pare{M_n-\chi_0 H_n }, \\
& \div\pare{H_n+ M_n}=\cJ_n F, \\
& \div \ u_n = \curl \ H_n =0, \\
& \left. \pare{u_n, \omega_n, M_n, H_n}\right|_{t=0}= \pare{\cJ_n u_0, \cJ_n\omega_0, \cJ_n M_0, \cJ_n H_0 }. 
\end{aligned}
\right.
\end{equation}

We want to rewrite the above system in a purely evolutionary form, such as it was for instance done in \cite{PZ12} for the $ Q $-tensor system.  To do so we have hence to incorporate th informations given by  the equations 
\begin{align*}
\div\pare{H_n+ M_n}=\cJ_n F && \text{and} && \div \ u_n = \curl \ H_n =0,
\end{align*}
in the evolution equation for $ u_n, \omega_n  $ and $ M_n $. Indeed using the relation \eqref{eq:H_in_fz_di_M} we can define $ H_n $ as a function of $ M_n $ and the external magnetic field as
\begin{equation}\label{eq:def_Hn}
H_n = -\cQ M_n +\cG_n, 
\end{equation}
where indeed $ \cG_n= \nabla\Delta^{-1} \cJ_n F $. Whence, denoting as $ \cP $ the Leray projector onto divergence-free vector fields, explicitly  defined as
\begin{equation*}
\begin{aligned}
\cP v &  =  \pare{1-\cQ}v, \\
 & = \pare{1- \Delta^{-1}\nabla\div} v, 
\end{aligned}
\end{equation*}
we can write the approximated system \eqref{eq:Rosensweig2D_approximated1} in a equivalent,  purely evolutionary form:

\begin{equation}\label{eq:Rosensweig2D_approximated}
\left\lbrace
\begin{aligned}
& 
\begin{multlined}
\rho_0 \pare{\Big. \partial_t u_n +\cP \cJ_n \pare{\cP u_n \cdot \nabla \cP u_n} } -\pare{\eta + \zeta} \Delta u_n \\
 = \mu_0 \  \cP \cJ_n \pare{\Big.   M_n\cdot \nabla \pare{-\cQ M_n +\cG_n}} + 2\zeta \cP \pare{\begin{array}{c}
\partial_2 \omega_n\\ -\partial_1\omega_n
\end{array}},
\end{multlined}
 \\[5mm]
&
\begin{multlined}
 \rho_0 \kappa \pare{\Big. \partial_t \omega_n + \cJ_n\pare{ \cP u_n\cdot \nabla \omega_n}} - \eta' \Delta \omega_n
 \\
 = \mu_0 \cJ_n \pare{ \Big.  M_n\times \pare{-\cQ M_n +\cG_n}} + 2\zeta \pare{ \curl \ \cP u_n -2\omega_n}, 
\end{multlined} 
 \\[5mm]
&
\begin{multlined}
 \partial_t M_n + \cJ_n\pare{ \cP u_n\cdot \nabla M_n} - \sigma \Delta M_n
\\
  = \cJ_n\pare{\pare{
\begin{array}{c}
-M_{2, n}\\ M_{1, n}
\end{array}
}\omega_n} -\frac{1}{\tau}\pare{M_n-\chi_0 \pare{-\cQ M_n +\cG_n} },
\end{multlined} 
 \\[5mm]
& \left. \pare{u_n, \omega_n, M_n}\right|_{t=0}= \pare{\cJ_n u_0, \cJ_n\omega_0, \cJ_n M_0 }. 
\end{aligned}
\right.
\end{equation}

Let us underline moreover that, being $ H_n $ defined via the equation \eqref{eq:def_Hn} the approximate magnetostatic equation in such setting reads as
\begin{equation*}
\begin{aligned}
\div\pare{H_n+M_n} & = \div\ M_n -\div\ \cQ M_n +\div\ \cG_n, 
\end{aligned}
\end{equation*}
but indeed $ \div\ M_n =\div\ \cQ M_n $, and since $ \cG_n = \Delta^{-1}\nabla \ \cJ_n F $ it is immediate that $ \div\ \cG_n = \cJ_n F $, whence we recover the fourth equation of \eqref{eq:Rosensweig2D_approximated1}. Moreover since by hypothesis $ \div \ u_0 =0 $ and the evolution equation of $ u_n $ can be written in the abstract form $ \partial_t u_n = \cP\ f \pare{u_n, x, t} $, for a suitable $ f $, it is hence assured that $ \div\ \bra{ u_n \pare{t}}=0 $ for any $ t $ in the (eventual) lifespan of $ u_n $.  \\

Let us now define the Hilbert space 
\begin{equation*}
H^{\sfrac{1}{2}}_n = \set{f\in H^{\frac{1}{2}}\pare{\bR^2} \ \left| \ \textnormal{Supp} \ \hat{f}\subset B_n\pare{0}\setminus B_{\sfrac{1}{n}}\pare{0}\right.  }, 
\end{equation*}
endowed with the $ H^{\frac{1}{2}}\pare{\bR^2} $ scalar product. \\

Indeed denoting $ V_n=\pare{u_n, \omega_n, M_n} $ we can say that the system \eqref{eq:Rosensweig2D_approximated} can be written in the autonomous form
\begin{equation*}
\frac{\d}{\d t} \ V_n = F_n\pare{V_n}, 
\end{equation*}
where $ F_n $ maps $ H^{\sfrac{1}{2}}_n $ onto itself. We can hence regard system \eqref{eq:Rosensweig2D_approximated} as an ordinary differential equation in $ H^{\sfrac{1}{2}}_n $ verifying the conditions of Cauchy-Lipschitz theorem. Being so we deduce the existence of a sequence of positive maximal lifespans $ \pare{T_n}_n $ such that, for each $ n $,  the system  \eqref{eq:Rosensweig2D_approximated} admits a unique maximal solution
\begin{equation*}
V_n \in \cC\pare{\left[0, T_n \right) ; H^{\sfrac{1}{2}}_n\cap H^{\infty}}.
\end{equation*}

The cancellation properties which allowed us to prove Lemma  \ref{lem:L2_energy_estimates} hold for the system \eqref{eq:Rosensweig2D_approximated} as well, whence following the same lines of the proofs of Lemma \ref{lem:L2_energy_estimates} and \ref{lem:H12_energy_estimates} we can prove similar uniform energy bounds, namely the following result holds true:

\begin{lemma}
\label{lem:energy_bounds_approximate_system}
Let $ u_0, \omega_0, M_0\in H^{\sfrac{1}{2}}, \ F\in L^{2}_{\loc}\pare{\bR_+; L^2}, \ \cG_F \in W^{1, \infty}_{\loc} \pare{\bR_+; H^{\sfrac{3}{2}}} $ and let $ \div \ u_0 =0 $. Fixed any  $ T>0 $  there exists a positive constant $ \mathbbl{c} =  \mathbbl{c}\pare{T, \rho_0, \eta, \zeta, \mu_0, \kappa, \eta' , \sigma ,  \tau } $ independent of $ n $ and  $ t\in\bra{0, T} $,  such that
\begin{align*}
\norm{\Big. \pare{ u_n, \omega_n, M_n}}_{L^\infty \pare{\bra{0, t}; H^{\frac{1}{2}}}} & \leqslant \mathbbl{c} , \\
\norm{\Big. \pare{ \nabla u_n, \nabla \omega_n, \nabla M_n}}_{L^2 \pare{\bra{0, t}; H^{\frac{1}{2}}}} &  \leqslant \frac{\mathbbl{c}}{c} \  , 
\end{align*}
where $ c $ is defined in \eqref{eq:def_c}. 
\end{lemma}

Considering hence the result stated in Lemma \ref{lem:energy_bounds_approximate_system},  fixed any positive, finite $ T > 0 $ we can  argue via a continuation argument in order to deduce that for each $ n $ the maximal lifespan of the unique solution of \eqref{eq:Rosensweig2D_approximated} is equal to  $ T $.  Since $ T $ is positive and arbitrary, we deduce that for each $ n $ 
 \begin{equation*}
V_n \in \cC \pare{ \bR_+ ; H^{\sfrac{1}{2}}_n}. 
\end{equation*}

\noindent Being moreover the energy bounds provided in Lemma \ref{lem:energy_bounds_approximate_system} uniform we can hence infer that the sequences $ \pare{V_n}_n, \pare{\nabla V_n}_n $ are, respectively, uniformly bounded in $ L^\infty_{\loc} \pare{\bR_+ ; H^{\sfrac{1}{2}} } $ and $ L^2_{\loc} \pare{\bR_+; H^{\sfrac{1}{2}} } $, from which we deduce, by interpolation, that the sequence 
\begin{equation*}
\pare{V_n}_n \ \text{is uniformly bounded in}\ L^2_{\loc} \pare{\bR_+; H^{\frac{3}{2}}\pare{\bR^2}}. 
\end{equation*}

This sort of uniform regularity is strong enough in order to provide a uniform bound for the sequence $ \pare{\partial_t V_n}_n $ in a space of the form $ L^2_{\loc}  \pare{\bR_+; H^\alpha} $, where $ \alpha < \sfrac{3}{2} $, possibly negative. \\

We already mentioned that the (approximated) Lorentz force $ \textsf{ F}^m_n = \mu_0 \  \cJ_n\pare{ \Big.  M_n \cdot \nabla\pare{-\cQ M_n +\cG_n}} $ is the less  regular term among all the nonlinear terms appearing in the system \eqref{eq:Rosensweig2D_approximated}, namely exploiting energy estimates only at a $ L^2 $ level (hence using the results provided by Lemma \ref{lem:L2_energy_estimates}) it is possible to prove only that $ \textsf{F}^m_n $ belongs (uniformly in $ n $) to $ L^1_{\loc}\pare{\bR_+; {H}^{-\frac{1}{2}}} $. This time regularity is not sufficiently strong in order to apply standard compactness results, such as Aubins-Lions lemma (cf. \cite{Aubin63}). For this reson hence we introduced in Lemma \ref{lem:H12_energy_estimates} the uniform $ \Hud $ energy estimates: such control on a higher level of derivatives will hence allow us to provide a uniform bound for the sequence $ \pare{\textsf{F}^m_n}_n $ in the $ L^2_{\loc} \pare{\bR_+; L^2} $ topology, making hence possible to apply Aubin-Lions lemma (cf. \cite{Aubin63}). \\

In order to bound the generic term $ \textsf{F}^m_n $ it suffice to remark that, given two vector fields $ \pare{a, b} \in H^{\frac{1}{2}}\pare{\bR^2}\times H^{\frac{1}{2}}\pare{\bR^2} $ the application
\begin{equation*}
\begin{aligned}
H^{\frac{1}{2}}\pare{\bR^2}\times H^{\frac{1}{2}}\pare{\bR^2} && \to && L^2 \pare{\bR^2}, \\
\pare{a, b} && \mapsto && a \otimes b, 
\end{aligned}
\end{equation*}
is continuous, hence we deduce that
\begin{equation*}
\begin{aligned}
\norm{\textsf{F}^m_n}_{L^2\pare{\bra{0, T}; L^2\pare{\bR^2}}} & \leqslant C \norm{M_n}_{L^\infty\pare{\bra{0, T};  H^{\frac{1}{2}}\pare{\bR^2} }}\norm{\nabla \pare{-\cQ M_n +\cG_n}}_{L^2\pare{\bra{0, T}; H^{\frac{1}{2}}\pare{\bR^2}  }},\\
& \leqslant C \norm{M_n}_{L^\infty\pare{\bra{0, T};  H^{\frac{1}{2}}\pare{\bR^2} }} \pare{\norm{\nabla  M_n}_{L^2\pare{\bra{0, T}; H^{\frac{1}{2}}\pare{\bR^2}  }}
+ \norm{\cG_n }_{W^{1, \infty}\pare{\bra{0, T}; H^{\frac{3}{2}}\pare{\bR^2}  }}
}, \\
& <\infty. 
\end{aligned} 
\end{equation*}
We implicitly used in the second inequality the fact that $ \cQ $ maps continuously $ H^{\sfrac{1}{2}} $ to itself. 
The right hand side of the above equation can be bounded, uniformly in $ n $,  in terms of the initial data thanks to the results stated in Lemma \ref{lem:energy_bounds_approximate_system}. \\

The remaining nonlinear approximate terms can be bounded using the global energy estimates at a $ L^2 $ level. Letting $ V_n = \pare{u_n, \omega_n, M_n} $ we remark that every nonlinear term appearing in \eqref{eq:Rosensweig2D_approximated} is either in the form
\begin{equation*}
u_n\cdot\nabla V_n, 
\end{equation*}
or either in the form 
\begin{equation*}
V_n\otimes V_n, 
\end{equation*}
whence it suffice to provide bounds for these two types of nonlinearities. \\

\noindent Being $ u_n $ divergence-free and using Lemma \ref{lem:prod_rules_Sobolev_2D} and an interpolation of Sobolev spaces we deduce the following bound
\begin{equation*}
\begin{aligned}
\norm{u_n \cdot \nabla V_n}_{L^2\pare{\bra{0, T}; H^{-1}}} & \leqslant \norm{u_n \otimes V_n}_{L^2\pare{\bra{0, T}; L^2}}, \\
& \leqslant \norm{ \norm{u_n}_{L^2}^{\sfrac{1}{2}} \norm{\nabla u_n}_{L^2}^{\sfrac{1}{2}}   \norm{V_n}_{L^2}^{\sfrac{1}{2}} \norm{\nabla V_n}_{L^2}^{\sfrac{1}{2}} }_{L^2\pare{\bra{0, T}}}, \\
 & \leqslant  \norm{u_n}_{L^\infty \pare{\bra{0, T};  L^2}}^{\sfrac{1}{2}} \norm{V_n}_{L^\infty \pare{\bra{0, T};  L^2}}^{\sfrac{1}{2}}
 \norm{\nabla u_n}_{L^2 \pare{\bra{0, T};  L^2}}^{\sfrac{1}{2}} \norm{\nabla V_n}_{L^2 \pare{\bra{0, T};  L^2}}^{\sfrac{1}{2}}<\infty, 
\end{aligned}
\end{equation*}
thanks to the results provided in Lemma \ref{lem:energy_bounds_approximate_system}. 
\\
While for the other kind of nonlinearity
\begin{equation*}
\begin{aligned}
\norm{V_n \otimes V_n}_{L^2\pare{\bra{0, T}; L^2}}
 & \leqslant  
\norm{V_n}_{L^\infty \pare{\bra{0, T} ; L^2}} \norm{\nabla V_n}_{L^2 \pare{\bra{0, T} ; L^2}} 
 <\infty. 
\end{aligned}
\end{equation*}

Applying hence Aubin-Lions lemma \cite{Aubin63} we state that the sequence 
\begin{equation*}
\pare{u_n, \omega_n, M_n}_{n\in\bN} \ \text{is compact in } \ L^2_{\loc} \pare{\bR_+; H^{\frac{3}{2}-\varepsilon}_{\loc} \pare{\bR^2}}, \ \forall \ \varepsilon > 0,
\end{equation*}
which implies that there exists at least one
\begin{equation*}
\pare{u, \omega, M}\in L^\infty_{\loc}\pare{\bR_+; H^{\frac{1}{2}}} \cap L^2_{\loc}\pare{\bR_+; H^{\frac{3}{2}}}, 
\end{equation*}
such  that the sequence $ \pare{u_n, \omega_n, M_n}_{n\in\bN} $ converges (up to relabeled subsequences if it may be) to
\begin{equation}\label{eq:conv_Leray}
\pare{u_n, \omega_n, M_n} \xrightarrow{n\to \infty} \pare{u, \omega, M} \text{ in } L^2_{\loc} \pare{\bR_+; H^{\frac{3}{2}-\varepsilon}_{\loc} \pare{\bR^2}}, \ \forall \varepsilon >0. 
\end{equation}
Such convergence is sufficiently strong in order to pass to the limit in the bilinear interactions in \eqref{eq:Rosensweig2D_approximated}. 

To prove that the limit element $ \pare{u, \omega, M} $ solves the system

\begin{equation}\label{eq:Rosensweig2D_limit}
\left\lbrace
\begin{aligned}
& 
\begin{multlined}
\rho_0 \pare{\Big. \partial_t u +\cP  \pare{\cP u \cdot \nabla \cP u} } -\pare{\eta + \zeta} \Delta u \\
 = \mu_0 \  \cP  \pare{\Big.   M\cdot \nabla \pare{-\cQ M +\cG_F}} + 2\zeta \cP \pare{\begin{array}{c}
\partial_2 \omega\\ -\partial_1\omega
\end{array}},
\end{multlined}
 \\[3mm]
&
\begin{multlined}
 \rho_0 \kappa \pare{\Big. \partial_t \omega + \pare{ \cP u\cdot \nabla \omega}} - \eta' \Delta \omega
 \\
 = \mu_0  \pare{ \Big.  M\times \pare{-\cQ M +\cG_F}} + 2\zeta \pare{ \curl \ \cP u -2\omega}, 
\end{multlined} 
 \\[5mm]
&
\begin{multlined}
 \partial_t M + \pare{ \cP u\cdot \nabla M} - \sigma \Delta M
\\
  = \pare{
\begin{array}{c}
-M_{2}\\ M_{1}
\end{array}
}\omega -\frac{1}{\tau}\pare{M-\chi_0 \pare{-\cQ M +\cG_F} },
\end{multlined} 
 \\[5mm]
& \left. \pare{u, \omega, M}\right|_{t=0}= \pare{ u_0, \omega_0, M_0 },
\end{aligned}
\right.
\end{equation}
in a weak sense is a rather standard procedure, facilitated by the rather high regularity of the convergence \eqref{eq:conv_Leray}. Among all the nonlinear interactions the convergence which is less immediate to prove is 
\begin{equation*}
M_n\cdot\nabla \cQ M_n \xrightarrow{n\to\infty} M\cdot\nabla \cQ M,.
\end{equation*}
Considering hence a test function $ \psi \in\cD\pare{\bR_+\times\bR^2} $
\begin{align*}
\int M_n\cdot\nabla \cQ M_n \ \psi \ \dx \d t- \int M\cdot\nabla \cQ M \ \psi \ \dx \d t & = \int \pare{M_n -M}\cdot \nabla \cQ M_n \ \psi \ \dx \d t + \int M\cdot\nabla \cQ \pare{M_n-M} \ \psi \ \dx \d t, \\
& = I_{1, n} + I_{2,n}. 
\end{align*}

Applying H\"older inequality we deduce the bound
\begin{align*}
I_{1, n} & \lesssim \norm{M_n - M}_{L^2_{\loc}\pare{\bR_+; L^4_{\loc}}} \norm{\nabla\cQ M_n}_{L^2_{\loc}\pare{\bR_+; L^2_{\loc}}} \norm{\psi}_{L^\infty\pare{\bR_+; L^4}}. 
\end{align*}
Standard Sobolev embeddings and \eqref{eq:conv_Leray} imply that
\begin{equation*}
\norm{M_n - M}_{L^2_{\loc}\pare{\bR_+; L^4_{\loc}}}\leqslant \norm{M_n - M}_{L^2_{\loc}\pare{\bR_+; H^{\sfrac{1}{2}}_{\loc}}}\to 0 \text{ as } n\to \infty, 
\end{equation*}
moreover since $ \nabla\cQ $ is  a pseudo-differential operator of order one we argue that
\begin{equation*}
\norm{\nabla\cQ M_n}_{L^2_{\loc}\pare{\bR_+; L^2_{\loc}}} \lesssim \norm{ M_n}_{L^2_{\loc}\pare{\bR_+; H^1_{\loc}}} < C < \infty, 
\end{equation*}
thanks to the results of Lemma \ref{lem:energy_bounds_approximate_system}, proving that $ I_{1, n}\xrightarrow{n\to \infty}0 $. \\
Similarly it can be proved that $ I_{2, n}\to 0 $ as $ n\to \infty $. \\

Defining hence 
\begin{equation*}
H = -\cQ M + \cG_F,
\end{equation*}
is straightforward to prove that considering $ \pare{u, \omega, M} $ the weak solution of \eqref{eq:Rosensweig2D_limit}, then $ \pare{u, \omega, M, H} $  solve weakly \eqref{eq:Rosensweig2D}.\\

The proof that the solutions constructed above are unique in the energy space $ L^\infty_{\loc}\pare{\bR_+; H^{\sfrac{1}{2}}} \cap L^2_{\loc}\pare{\bR_+; H^{\sfrac{3}{2}}} $ is postponed in Appendix \ref{sec:uniquess_weak}. 
\end{proof}

\section{Propagation of higher regularity} \label{sec:HE}

In this section we prove that given any initial data in $ H^k, k \in \bN $ it is hence possible to  construct global strong solutions for the system \eqref{eq:Rosensweig2D}. The result we prove is the following one

\begin{prop}\label{prop:propagation_higher_regularity}
Let $ k\geqslant 1 $,  $ u_0, \omega_0, M_0, H_0 \in H^k\pare{\bR^2} $  such that $ \div \ u_0 =0 $ and let $ \cG_F\in W^{1, \infty}_{\loc}\pare{\bR_+; H^{k+1}} $. The unique global weak solution $ U =\pare{u, \omega, M , H} $ of \eqref{eq:Rosensweig2D} identified in Proposition \ref{prop:existence_global_weak_H12_solutions}   enjoys the following additional regularity
\begin{equation*}
U\in \cC\pare{\bR_+; H^k \pare{\bR^2} }, \hspace{1cm} \nabla U \in L^2_{\loc} \pare{\bR_+; H^k\pare{\bR^2} }. 
\end{equation*}
\end{prop}

\begin{rem}
Indeed if $ k>1+\ell + \rho $ where $ \ell \in\bN $ and $ \rho\in\bra{0, 1} $ the space $ H^k\pare{\bR^2} $ embeds continuously  in the H\"older space $ \cC^{\ell, \rho}\pare{\bR^2} $, whence Proposition \ref{prop:propagation_higher_regularity} implies the propagation of smoothness. \fine
\end{rem}

\begin{rem}\label{rem:propagation_HE}
In Proposition \ref{prop:existence_global_weak_H12_solutions} we proved that $ U $ weak solutions of \eqref{eq:Rosensweig2D} is such that
\begin{equation*}
U \in L^2_{\loc} \pare{\bR_+; H^{\frac{3}{2}}\pare{\bR^2}}, 
\end{equation*}
whence considering the embedding $ H^{\frac{3}{2}}\pare{\bR^2} \hra L^\infty \pare{\bR^2} $ we conclude that $ U \in L^2_{\loc} \pare{\bR_+; L^\infty \pare{\bR^2}} $.\\

Let us now consider $ u $ to be a solution of the two-dimensional incompressible \NS\ equations. If we suppose $ u $ to be in the space $ L^2_{\loc} \pare{\bR_+; L^\infty \pare{\bR^2}} $ it is rather easy to deduce that such regularity for the unknown $ u $ is sufficient in order to immediately deduce global propagation of $ H^k\pare{\bR^2}, \ k\in\bN, \ k\geqslant 1 $ regularity for the \NS\ incompressible equations. Being in fact true that
\begin{equation*}
\norm{a \ b}_{H^k} \lesssim \norm{a}_{H^k} \norm{b}_{L^\infty} + \norm{a}_{L^\infty} \norm{b}_{H^k}, 
\end{equation*}
for a proof of which we refer to \cite[Corollary 2.54, p. 90]{BCD}, we can argue that
\begin{equation}
\label{eq:estimate_NSI}
\begin{aligned}
\av{\ps{u\cdot\nabla u}{u}_{H^k}}& \leqslant \av{\ps{u \otimes u}{\nabla u}_{H^k}}, \\
& \leqslant \frac{\nu}{2}\norm{\nabla u}_{H^k}^2 + C \norm{u}_{L^\infty}^2 \norm{u}_{H^k}^2, 
\end{aligned}
\end{equation}
and such kind of term can be absorbed and controlled with standard parabolic energy estimates. \\
\noindent
For the system \eqref{eq:Rosensweig2D} though the global propagation of high-order regularity is not such an immediate deduction. We explain such defect of regularity considering only the perturbations generated by the Lorentz force
\begin{equation*}
\textsf{F}^m = \mu_0 \ M\cdot \nabla H.
\end{equation*} 
Let us recall in fact that in the Lorentz force the vector field $ M $ is \textit{not divergence-free}, hence the commutation of derivatives performed in \eqref{eq:estimate_NSI} cannot be done in such setting. Performing similar computations as the ones above one convinces himself that, setting $ W =\pare{M, H} $, one must have at least $ W\in L^2_{\loc} \pare{\bR_+; L^\infty \pare{\bR^2}} \cap L^1_{\loc}\pare{\bR_+;W^{1, \infty}\pare{\bR^2}} $ (which amounts to require $ W\in L^2_{\loc} \pare{\bR_+; H^{1+\eta} \pare{\bR^2}} \cap L^1_{\loc}\pare{\bR_+; H^{2+\eta}\pare{\bR^2}}, \eta >0 $ in terms of Sobolev regularity\footnote{Sharper criteria could be deduced.}) in order mimic the procedure explained above for the incompressible \NS\ equations.   \fine
\end{rem}
\hspace{5mm}

Indeed in order to prove Proposition \ref{prop:propagation_higher_regularity} it suffice to perform some $ H^k $ energy estimates on the system \eqref{eq:Rosensweig2D}. In this spirit the following auxiliary lemma will be used:

\begin{lemma}\label{lem:higher_energy_est_transport_term}
Let $ a, b, c \in L^\infty_{\loc} \pare{\bR_+; H^k \pare{\bR^2} } $ and $ \nabla a, \nabla b, \nabla c \in L^2_{\loc} \pare{\bR_+; H^k\pare{\bR^2} } $ be such that for each $ j\in\set{0, \ldots, k-1} $ there exists two functions $ \pare{F_j, G_j}\in L^\infty_{\loc}\pare{\bR_+} \times  L^2_{\loc}\pare{\bR_+} $
  such that
 \begin{equation}
 \label{eq:technical_bound}
\begin{aligned}
\norm{\Big. \pare{a, b}\pare{t}}_{H^j} & \leqslant F_j\pare{t},\\
\norm{\Big. \nabla \pare{a, b}\pare{t}}_{H^j} & \leqslant G_j\pare{t},\end{aligned}
 \end{equation}
then for each $ \epsilon > 0 $ and each multi-index $ \alpha $ such that $ \av{\alpha}=k $ the following inequality holds true
\begin{multline}\label{eq:bound_Leibniz_rule}
\av{ \int \pa \pare{a\cdot \nabla b}\ \cdot \ \pa c \ \dx} \leqslant \epsilon \pare{\norm{\nabla \pa  a}_{L^2}^2 + \norm{\nabla \pa  b}_{L^2}^2 + \norm{\nabla \pa  c}_{L^2}^2 }\\
+ \frac{C_k}{\epsilon} \ \pare{\sum_{\ell=0}^{k-1}F_\ell^2 G_\ell^2 
+ G_0^2 } \norm{ \pa  c}_{L^2}^2 +\frac{C_k}{\epsilon}\ G_0^2 \norm{ \pa  a}_{L^2}^2
+ \frac{C_k}{\epsilon} \sum_{\ell=1}^{k-1} G_\ell^2, 
\end{multline}
where $ \pa = \partial_1^{\alpha_1}\partial_2^{\alpha_2} $. 
\end{lemma}

The proof of Lemma \ref{lem:higher_energy_est_transport_term} is postponed to Appendix \ref{sec:pf_lem_higher_energy_est_transport_term} for the sake of readability. \\

In an analogous way as we prove Lemma \ref{lem:higher_energy_est_transport_term} we can prove the following result

\begin{lemma}\label{lem:higher_energy_est_bilinear_term}
Let $ a, b, c $ satisfy the same hypothesis as in Lemma \ref{lem:higher_energy_est_transport_term}, then
\begin{multline} \label{eq:bilinear_bound_high_energy}
\av{\int \pa \pare{a \  b} \ \pa c \ \dx}  \leqslant
\epsilon \pare{ \norm{\nabla \pa a}_{L^2}^2 + \norm{\nabla \pa b}_{L^2}^2 + 2\norm{\nabla \pa c}_{L^2}^2 }
\\
+ C_k \pare{ 
\sum_{\ell = 1}^{k-1} G_{\ell-1}^{\sfrac{1}{2}} G_{\ell}^{\sfrac{1}{2}} G_{k-\ell-1}^{\sfrac{1}{2}} G_{k-\ell}^{\sfrac{1}{2}}
} \norm{\pa c}_{L^2} 
+
\frac{C_k}{\epsilon} \ G_{k-1}^2 \pare{\norm{a}_{L^2}^2 + \norm{b}_{L^2}^2 }. 
\end{multline}
\end{lemma}

The proof is postponed in Appendix \ref{sec:pf_lem_higher_energy_est_bilinear_term}. \\

\textit{Proof of Proposition \ref{prop:propagation_higher_regularity} :} 
If $ k=0 $ Proposition \ref{prop:existence_global_weak_H12_solutions} proves the claim, hence without loss of generality we can assume that $ k>0 $. \\

We  prove the claim with an inductive hypothesis on $ k $, since, as explained above, we can assume that $ k \geqslant 1 $. We assume that:

\begin{hyp}\label{hyp:induction}
For each $ j\in \set{0, \ldots, k-1}  $ we suppose  $ U\in \cC\pare{\bR_+; H^j \pare{\bR^2} } $ and $ \nabla U \in L^2_{\loc} \pare{\bR_+; H^j\pare{\bR^2} } $. Moreover for each $ j\in \set{0, \ldots, k-1}  $ there exist two functions $ F_j $ and $ G_j $ which are respectively $ L^\infty_\loc\pare{\bR_+} $ and $ L^2_{\loc} \pare{\bR_+} $ such that, for each $ t\geqslant 0 $
\begin{align*}
\norm{U\pare{t}}_{H^j } \leqslant F_j\pare{t}, && \norm{\nabla U\pare{t}}_{H^j } \leqslant G_j\pare{t}. 
\end{align*}
\fine
\end{hyp}

Let us hence consider a multi-index $ \alpha =\pare{\alpha_1, \alpha_2} $ such that $ \av{\alpha}=k $, let us apply the operator $ \partial^\alpha $ to the equation of $ u $ in \eqref{eq:Rosensweig2D}, let us multiply the resulting equation for $ \partial^\alpha u $ and let integrate in space obtaining the inequality
\begin{multline}\label{eq:bu1}
\frac{\rho_0}{2}\ \frac{\d}{\d t} \norm{\partial^\alpha u }_{L^2}^2 + \pare{\eta + \zeta} \norm{\nabla \partial^\alpha u }_{L^2}^2 
\\
\leqslant \av{\ps{\pa   \pare{u\cdot \nabla u}}{\pa u}_{L^2}} + \mu_0 \av{\ps{\pa \pare{M\cdot \nabla H}}{u}_{L^2}} +  2\zeta \av{\ps{\pare{\begin{array}{c}
\partial_2 \pa \omega\\ -\partial_1\pa \omega
\end{array}}}{\pa u}_{L^2}}. 
\end{multline}

We hence apply the inequality \eqref{eq:bound_Leibniz_rule} to the terms $ \av{\ps{\pa   \pare{u\cdot \nabla u}}{\pa u}_{L^2}} $ and $ \mu_0 \av{\ps{\pa \pare{M\cdot \nabla H}}{u}_{L^2}} $ obtaining the bounds
\begin{align}
\label{eq:bu2}
\av{\ps{\pa   \pare{u\cdot \nabla u}}{\pa u}_{L^2}} & \leqslant \epsilon \norm{\nabla \pa u}_{L^2}^2 + \frac{C}{\epsilon} \pare{\sum_{\ell=0}^{k-1}F_\ell^2 G_\ell^2 
+ G_0^2 } \norm{ \pa  u}_{L^2}^2 + \sum_{\ell=1}^{k-1} G_\ell^2, \\
\mu_0 \av{\ps{\pa \pare{M\cdot \nabla H}}{u}_{L^2}} & \leqslant  \epsilon \pare{\norm{\nabla \pa  M}_{L^2}^2 + \norm{\nabla \pa  H}_{L^2}^2 + \norm{\nabla \pa  u}_{L^2}^2 }\nonumber \\
& + \frac{C}{\epsilon}\bra{ \ \pare{\sum_{\ell=0}^{k-1}F_\ell^2 G_\ell^2 
+  G_0^2 } \norm{ \pa  u}_{L^2}^2 +\ G_0^2 \norm{ \pa  M}_{L^2}^2}
+ \frac{C}{\epsilon} \sum_{\ell=1}^{k-1} G_\ell^2.\nonumber
\end{align}

But since $ H = -\cQ M +\cG_F $ we can apply Lemma \ref{lem:reg_H} to deduce that
\begin{equation*}
\norm{\nabla \pa  H}_{L^2}^2 \lesssim \norm{\nabla \pa  M}_{L^2}^2 + \norm{\nabla \pa  \cG_F}_{L^2}^2, 
\end{equation*}
whence
\begin{multline}\label{eq:bu3}
\mu_0 \av{\ps{\pa \pare{M\cdot \nabla H}}{u}_{L^2}}  \leqslant
\epsilon \pare{2\norm{\nabla \pa  M}_{L^2}^2  + \norm{\nabla \pa  u}_{L^2}^2 } \\
+ \frac{C}{\epsilon}\bra{ \ \pare{\sum_{\ell=0}^{k-1}F_\ell^2 G_\ell^2 
+  G_0^2 } \norm{ \pa  u}_{L^2}^2 +\ G_0^2 \norm{ \pa  M}_{L^2}^2} +
\bra{
\frac{C}{\epsilon} \sum_{\ell=1}^{k-1} G_\ell^2 + \epsilon \norm{\nabla \pa  \cG_F}_{L^2}^2
}. 
\end{multline}

Moreover a simple Cauchy-Schwartz inequality imply
\begin{equation*}
\av{\ps{\pare{\begin{array}{c}
\partial_2 \pa \omega\\ -\partial_1\pa \omega
\end{array}}}{\pa u}_{L^2}}\leqslant \epsilon \norm{\nabla\pa  \omega}_{L^2}^2 + \frac{C}{\epsilon} \norm{\pa u}_{L^2}^2, 
\end{equation*}
but indeed, considering the inductive hypothesis, Hypothesis \ref{hyp:induction}, 
\begin{equation*}
\norm{\pa u}_{L^2}^2 \leqslant\norm{\nabla u }_{H^{k-1}}^2 \leqslant G_{k-1}^2 ,
\end{equation*}
whence
\begin{equation}\label{eq:bu4}
\av{\ps{\pare{\begin{array}{c}
\partial_2 \pa \omega\\ -\partial_1\pa \omega
\end{array}}}{\pa u}_{L^2}}
\leqslant
\epsilon \norm{\nabla\pa  \omega}_{L^2}^2 + \frac{C}{\epsilon} \ G_{k-1}^2 . 
\end{equation}
Whence inserting the bounds \eqref{eq:bu2}, \eqref{eq:bu3}, \eqref{eq:bu4} in \eqref{eq:bu1} we deduce the inequality

\begin{multline} \label{eq:HE_bound_u}
\frac{\rho_0}{2}\ \frac{\d}{\d t} \norm{\partial^\alpha u }_{L^2}^2 + \pare{\eta + \zeta} \norm{\nabla \partial^\alpha u }_{L^2}^2 
\leqslant 
\epsilon \pare{ \norm{\nabla\pa  \omega}_{L^2}^2 + 2\norm{\nabla \pa  M}_{L^2}^2  + \norm{\nabla \pa  u}_{L^2}^2 } \\
+ \frac{C}{\epsilon}\bra{ \ \pare{\sum_{\ell=0}^{k-1}F_\ell^2 G_\ell^2 
+  G_0^2 } \norm{ \pa  u}_{L^2}^2 +\ G_0^2 \norm{ \pa  M}_{L^2}^2} +
\bra{
\frac{C}{\epsilon} \sum_{\ell=1}^{k-1} G_\ell^2 + \epsilon \norm{\nabla \pa  \cG_F}_{L^2}^2 
}.
\end{multline}

Let us now perform the same kind of procedure on the equation describing the evolution for $ \omega  $ in \eqref{eq:Rosensweig2D}, we deduce
\begin{multline}\label{eq:bo1}
\frac{\rho_0\kappa}{2}\ \frac{\d}{\d t} \norm{\pa \omega}_{L^2}^2 + \eta' \norm{\nabla \pa \omega}_{L^2}^2 \leqslant \rho_0\kappa\av{\ps{ \pa \pare{u\cdot\nabla \omega}}{\pa \omega}_{L^2}} + \mu_0 \av{\ps{\pa \pare{M\times H}}{\pa \omega}_{L^2}}\\
+2\zeta\av{\ps{\pa\curl\ u}{\pa \omega}_{L^2}} - 4\zeta \norm{\pa \omega}_{L^2}^2. 
\end{multline}

Using the estimates \eqref{eq:bound_Leibniz_rule} and \eqref{eq:bilinear_bound_high_energy} respectively we can estimate the following terms
\begin{align}
&\begin{multlined}
\rho_0\kappa\av{\ps{ \pa \pare{u\cdot\nabla \omega}}{\pa \omega}_{L^2}} \leqslant
\frac{\epsilon}{8} \ \pare{\norm{\nabla \pa  u}_{L^2}^2 + 2 \norm{\nabla \pa  \omega}_{L^2}^2 }\\[5mm]
+ \frac{C}{\epsilon} \ \pare{\sum_{\ell=0}^{k-1}F_\ell^2 G_\ell^2 
+ G_0^2 } \norm{ \pa  \omega}_{L^2}^2 +\frac{C}{\epsilon}\ G_0^2 \norm{ \pa  u}_{L^2}^2
+ \frac{C}{\epsilon} \sum_{\ell=1}^{k-1} G_\ell^2, 
\end{multlined} \label{eq:bo2} \\[5mm]
& \begin{multlined}
\mu_0 \av{\ps{\pa \pare{M\times H}}{\pa \omega}_{L^2}} 
\leqslant 
\epsilon \pare{ \norm{\nabla \pa M}_{L^2}^2 + \norm{\nabla \pa H}_{L^2}^2 + \norm{\nabla \pa \omega}_{L^2}^2 }
\\[5mm]
+ C \pare{ 
\sum_{\ell = 1}^{k-1} G_{\ell-1}^{\sfrac{1}{2}} G_{\ell}^{\sfrac{1}{2}} G_{k-\ell-1}^{\sfrac{1}{2}} G_{k-\ell}^{\sfrac{1}{2}}
} \norm{\pa \omega}_{L^2} 
+
\frac{C}{\varepsilon} \ G_{k-1}^2 \pare{\norm{M}_{L^2}^2 + \norm{H}_{L^2}^2 }. 
\end{multlined}\nonumber
\end{align}

Using Lemma \ref{lem:reg_H} we can state that

\begin{align*}
\norm{\nabla \pa  H}_{L^2}^2 & \lesssim \norm{\nabla \pa  M}_{L^2}^2 + \norm{\nabla \pa  \cG_F}_{L^2}^2,\\
G_{k-1}^2 \norm{H}_{L^2}^2 & \lesssim G_{k-1}^2 \norm{M}_{L^2}^2 + G_{k-1}^2 \norm{\cG_F}_{L^2}^2, 
\end{align*}
whence
\begin{multline} \label{eq:bo3}
\mu_0 \av{\ps{\pa \pare{M\times H}}{\pa \omega}_{L^2}} 
\leqslant 
\frac{\epsilon}{8} \  \pare{ 2 \norm{\nabla \pa M}_{L^2}^2  + 2 \norm{\nabla \pa \omega}_{L^2}^2 }
\\
\hspace{2cm}
+ C \pare{ 
\sum_{\ell = 1}^{k-1} G_{\ell-1}^{\sfrac{1}{2}} G_{\ell}^{\sfrac{1}{2}} G_{k-\ell-1}^{\sfrac{1}{2}} G_{k-\ell}^{\sfrac{1}{2}}
} \norm{\pa \omega}_{L^2} 
+
\frac{C}{\varepsilon} \ G_{k-1}^2 \pare{2\norm{M}_{L^2}^2 + \norm{\cG_F}_{L^2}^2 }  + \epsilon \norm{\nabla \pa  \cG_F}_{L^2}^2. 
\end{multline}
While
\begin{equation}\label{eq:bo4}
2\zeta\av{\ps{\pa\curl\ u}{\pa \omega}_{L^2}} \leqslant \epsilon \norm{\nabla\pa  u }_{L^2}^2 + \frac{C}{\epsilon} \ G_{k-1}^2, 
\end{equation}
as it was argued in order to prove \eqref{eq:bu4}. Whence the bounds \eqref{eq:bo2}, \eqref{eq:bo3} and \eqref{eq:bo4} transform \eqref{eq:bo1} in

\begin{multline} \label{eq:HE_bound_omega}
\frac{\rho_0\kappa}{2}\ \frac{\d}{\d t} \norm{\pa \omega}_{L^2}^2 + \eta' \norm{\nabla \pa \omega}_{L^2}^2 
+ 4\zeta \norm{\pa \omega}_{L^2}^2 \leqslant  
{\tilde{\epsilon}} \ \pare{\norm{\nabla \pa  u}_{L^2}^2 + \norm{\nabla \pa  M}_{L^2}^2 + \norm{\nabla \pa  \omega}_{L^2}^2 } \\
+ \frac{C}{\tilde{\epsilon}} \ \pare{\sum_{\ell=0}^{k-1}F_\ell^2 G_\ell^2 
+ G_0^2 } \norm{ \pa  \omega}_{L^2}^2 +\frac{C}{\tilde{\epsilon}}\ G_0^2 \norm{ \pa  u}_{L^2}^2
+ C \pare{ 
\sum_{\ell = 1}^{k-1} G_{\ell-1}^{\sfrac{1}{2}} G_{\ell}^{\sfrac{1}{2}} G_{k-\ell-1}^{\sfrac{1}{2}} G_{k-\ell}^{\sfrac{1}{2}}
} \norm{\pa \omega}_{L^2} 
\\
+ \frac{C}{\tilde{\epsilon}} \pare{ \sum_{\ell=1}^{k-1} G_\ell^2 + G_{k-1}^2 \pare{2\norm{M}_{L^2}^2 + \norm{\cG_F}_{L^2}^2 }  + \tilde{\epsilon} \norm{\nabla \pa  \cG_F}_{L^2}^2},
\end{multline}
where $ \tilde{\epsilon}<\sfrac{\epsilon}{16} $.\\

\noindent
Applying the operator $ \pa $ to the equation of $ M $ in \eqref{eq:Rosensweig2D}, multiplying for $ \pa M $ and integrating in the variable $ x\in\bR^2 $ give us instead
\begin{multline}
\label{eq:bM1}
\frac{1}{2}\ \frac{\d}{\d t} \norm{\pa M}_{L^2}^2 + \sigma \norm{\nabla M}_{L^2}^2 + \frac{1}{\tau} \norm{\pa M}_{L^2}^2 \leqslant \av{ \ps{\pa \pare{u\cdot\nabla M}}{\pa M}_{L^2}} \\
+ \av{\ps{\pa \pare{ \pare{
\begin{array}{c}
-M_2\\ M_1
\end{array}
}\omega}}{\pa M}_{L^2}}
+ \frac{\chi_0}{\tau} {\ps{\pa H}{\pa M}_{L^2}}. 
\end{multline}

Using the estimate \eqref{eq:bound_Leibniz_rule} and \eqref{eq:bilinear_bound_high_energy} we deduce the bounds
\begin{align}
& \begin{multlined}
\rho_0\kappa\av{\ps{ \pa \pare{u\cdot\nabla M}}{\pa M }_{L^2}} \leqslant
\frac{\epsilon}{8} \ \pare{\norm{\nabla \pa  u}_{L^2}^2 + 2 \norm{\nabla \pa  M }_{L^2}^2 }\\[5mm]
\hspace{2cm}
+ \frac{C}{\epsilon} \ \pare{\sum_{\ell=0}^{k-1}F_\ell^2 G_\ell^2 
+ G_0^2 } \norm{ \pa  M }_{L^2}^2 +\frac{C}{\epsilon}\ G_0^2 \norm{ \pa  u}_{L^2}^2
+ \frac{C}{\epsilon} \sum_{\ell=1}^{k-1} G_\ell^2, 
\end{multlined} 
\label{eq:bM2}
\\[5mm]
& \begin{multlined}
\av{\ps{\pa \pare{ \pare{
\begin{array}{c}
-M_2\\ M_1
\end{array}
}\omega}}{\pa M}_{L^2}}
\leqslant 
\frac{\epsilon}{8} \ \pare{ 2 \norm{\nabla \pa M}_{L^2}^2  + \norm{\nabla \pa \omega}_{L^2}^2 }
\\[5mm]
+ C \pare{ 
\sum_{\ell = 1}^{k-1} G_{\ell-1}^{\sfrac{1}{2}} G_{\ell}^{\sfrac{1}{2}} G_{k-\ell-1}^{\sfrac{1}{2}} G_{k-\ell}^{\sfrac{1}{2}}
} \norm{\pa M }_{L^2} 
+
\frac{C}{\varepsilon} \ G_{k-1}^2 \pare{\norm{M}_{L^2}^2 + \norm{\omega}_{L^2}^2 }. 
\end{multlined}
\label{eq:bM3}
\end{align}

Using the identity \eqref{eq:H_in_fz_di_M} and the fact that the operator $ \cQ $ commutes with the operator $ \pa $
\begin{equation}
\label{eq:bM4}
\begin{aligned}
\frac{\chi_0}{\tau} {\ps{\pa H}{\pa M}_{L^2}} & = - \frac{\chi_0}{\tau} {\ps{\pa \cQ M}{\pa M}_{L^2}} + \frac{\chi_0}{\tau} {\ps{\pa \cG_F}{\pa M}_{L^2}}, \\
&=  - \frac{ \chi_0}{2 \tau} \norm{\pa \cQ M}_{L^2}^2 +  \frac{C}{\tau} \norm{\pa \cG_F}_{L^2}^2. 
\end{aligned}
\end{equation}

Whence inserting \eqref{eq:bM2}, \eqref{eq:bM3} and \eqref{eq:bM4} in \eqref{eq:bM1} we deduce
\begin{multline}\label{eq:HE_bound_M}
\frac{1}{2}\ \frac{\d}{\d t} \norm{\pa M}_{L^2}^2 + \sigma \norm{\nabla M}_{L^2}^2 + \frac{1}{\tau} \pare{ \norm{\pa M}_{L^2}^2 + \frac{ \chi_0}{2 } \norm{\pa \cQ M}_{L^2}^2}\\
 \leqslant 
 {\tilde{\epsilon}} \ \pare{\norm{\nabla \pa  u}_{L^2}^2 + \norm{\nabla \pa  M }_{L^2}^2 + \norm{\nabla \pa  \omega}_{L^2}^2} 
 +
 \frac{C}{\tilde{\epsilon}} \ \pare{\sum_{\ell=0}^{k-1}F_\ell^2 G_\ell^2 
+ G_0^2 } \norm{ \pa  M }_{L^2}^2 +\frac{C}{\tilde{\epsilon}}\ G_0^2 \norm{ \pa  u}_{L^2}^2
\\
\pare{ 
\sum_{\ell = 1}^{k-1} G_{\ell-1}^{\sfrac{1}{2}} G_{\ell}^{\sfrac{1}{2}} G_{k-\ell-1}^{\sfrac{1}{2}} G_{k-\ell}^{\sfrac{1}{2}}
} \norm{\pa M }_{L^2} 
+ C \set{\frac{1}{\tilde{\epsilon}} \bra{ \sum_{\ell=1}^{k-1} G_\ell^2 +G_{k-1}^2 \pare{\norm{M}_{L^2}^2 + \norm{\omega}_{L^2}^2 }}+\frac{1}{\tau} \norm{\pa \cG_F}_{L^2}^2}, 
\end{multline}
where $ \tilde{\epsilon}<\sfrac{\epsilon}{16} $.\\

\noindent
Whence adding \eqref{eq:HE_bound_u}, \eqref{eq:HE_bound_omega} and \eqref{eq:HE_bound_M}, and denoting
\begin{equation}
\label{eq:def_fis}
\begin{aligned}
& f_2  = \pare{\sum_{\ell=0}^{k-1}F_\ell^2 G_\ell^2 
+ G_0^2 } , \\
& f_1 = \sum_{\ell = 1}^{k-1} G_{\ell-1}^{\sfrac{1}{2}} G_{\ell}^{\sfrac{1}{2}} G_{k-\ell-1}^{\sfrac{1}{2}} G_{k-\ell}^{\sfrac{1}{2}}, \\
& f_0 = { \sum_{\ell=1}^{k-1} G_\ell^2 +G_{k-1}^2 \pare{ \norm{U}_{L^2}^2 + \norm{\cG_F }_{L^2}^2  }}+\frac{1}{\tau} \norm{ \cG_F}_{H^k }^2 + \norm{\nabla  \cG_F}_{H^k}^2, 
\end{aligned}
\end{equation}
where as usual $ U=\pare{u, \omega, M, H} $
we recover the inequality
\begin{multline*}
\frac{1}{2}\ \frac{\d}{\d t} \pare{\rho_0 \norm{\pa u}_{L^2}^2 + \rho_0 \kappa \norm{\pa \omega}^2 + \norm{\pa M}_{L^2}^2}\\
 + \pare{ \min\set{\eta+\zeta, \eta', \sigma \Big. }-\epsilon} \pare{\norm{\nabla \pa  u}_{L^2}^2 
 + \norm{\nabla \pa  \omega}_{L^2}^2 + \norm{\nabla \pa  M }_{L^2}^2}\\
  \leqslant \frac{C}{\epsilon} \ f_2 \  \pare{\rho_0 \norm{\pa u}_{L^2}^2 + \rho_0 \kappa \norm{\pa \omega}^2 + \norm{\pa M}_{L^2}^2} + C\  f_1 \pare{\Big. \rho_0 \norm{\pa u}_{L^2} + \rho_0 \kappa \norm{\pa \omega} + \norm{\pa M}_{L^2}}
  +\frac{C}{\epsilon} \ f_0. 
\end{multline*}

Setting hence $ \epsilon $ sufficiently small so that
\begin{equation*}
\min\set{\eta+\zeta, \eta', \sigma \Big. }-\epsilon \geqslant c > 0, 
\end{equation*}
and applying a Gronwall inequality we deduce hence that
\begin{multline}\label{eq:HE_Gronwall_applied_1}
\rho_0 \norm{\pa u \pare{t}}_{L^2}^2 + \rho_0 \kappa \norm{\pa \omega \pare{t}}^2 + \norm{\pa M \pare{t}}_{L^2}^2
\\
+ c \int_0^t \bra{  \norm{\nabla \pa  u\pare{t'}}_{L^2}^2 
 + \norm{\nabla \pa  \omega\pare{t'}}_{L^2}^2 + \norm{\nabla \pa  M\pare{t'} }_{L^2}^2 }\exp\set{C \int _{t'}^t f_2\pare{t''} + f_1\pare{t''}\d t''}\d t'
 \\
 \leqslant
 C \pare{\rho_0 \norm{\pa u _0}_{L^2}^2 + \rho_0 \kappa \norm{\pa \omega_0}^2 + \norm{\pa M_0}_{L^2}^2} \exp\set{C \int_0^t f_2\pare{t'} + f_1\pare{t'}\d t'}\\
 + C\int_0^t \bra{\Big. f_1\pare{t'} + f_0\pare{t'}} \exp\set{C \int _{t'}^t f_2\pare{t''} + f_1\pare{t''}\d t''} \d t'. 
\end{multline}

\noindent The right-hand side of \eqref{eq:HE_Gronwall_applied_1} is bounded in compact sets of $ \bR_+ $ considering the definition of the $ f_i $'s functions given in \eqref{eq:def_fis} and the regularity of the functions $ F_i, G_i $ set in the inductive hypothesis, Hypothesis \ref{hyp:induction}. Setting hence 
\begin{equation*}
\delta <\frac{c}{\min\set{\Big. \rho_0, \rho_0 \kappa, 1}}, 
\end{equation*}
and since for each $ 0\leqslant t'\leqslant t \leqslant T $
\begin{equation*}
\exp\set{C \int _{t'}^t f_2\pare{t''} + f_1\pare{t''}\d t''} \geqslant 1, 
\end{equation*}
we can transform \eqref{eq:HE_Gronwall_applied_1} into
\begin{multline}\label{eq:HE_Gronwall_applied_2}
\pare{\norm{\pa u \pare{t}}_{L^2}^2 +  \norm{\pa \omega \pare{t}}^2 + \norm{\pa M \pare{t}}_{L^2}^2}\\
+ \delta \int_0^t \bra{  \norm{\nabla \pa  u\pare{t'}}_{L^2}^2 
 + \norm{\nabla \pa  \omega\pare{t'}}_{L^2}^2 + \norm{\nabla \pa  M\pare{t'} }_{L^2}^2 }\d t'
 \\
 \leqslant
 C \pare{\rho_0 \norm{\pa u _0}_{L^2}^2 + \rho_0 \kappa \norm{\pa \omega_0}^2 + \norm{\pa M_0}_{L^2}^2} \exp\set{C \int_0^t f_2\pare{t'} + f_1\pare{t'}\d t'}\\
 + C\int_0^t \bra{\Big. f_1\pare{t'} + f_0\pare{t'}} \exp\set{C \int _{t'}^t f_2\pare{t''} + f_1\pare{t''}\d t''} \d t'. 
\end{multline}

Since the functions $ f_i $'s defined in \eqref{eq:def_fis} are independent of the choice of the multi-index $ \alpha $ we can sum each inequality derived in \eqref{eq:HE_Gronwall_applied_2} on the set of multi-indexes $ \alpha $ of length $ k $ deriving the inequality (here we denote $ V =\pare{u, \omega, M} $)
\begin{multline*}
\norm{V \pare{t}}_{H^k}^2 + \delta \int_0^t \norm{\nabla V \pare{t'}}^2_{H^k}\d t' \leqslant C \norm{U_0}_{H^k}^2\exp\set{C \int_0^t f_2\pare{t'} + f_1\pare{t'}\d t'}\\
 + C\int_0^t \bra{\Big. f_1\pare{t'} + f_0\pare{t'}} \exp\set{C \int _{t'}^t f_2\pare{t''} + f_1\pare{t''}\d t''} \d t'. 
\end{multline*}

We can hence set
\begin{align*}
& \begin{multlined}
\widetilde{F}_k^2\pare{t} = C \norm{U_0}_{H^k}^2\exp\set{C \int_0^t f_2\pare{t'} + f_1\pare{t'}\d t'}\\
\hspace{2cm} + C\int_0^t \bra{\Big. f_1\pare{t'} + f_0\pare{t'}} \exp\set{C \int _{t'}^t f_2\pare{t''} + f_1\pare{t''}\d t''} \d t', 
\end{multlined}\\
& \widetilde{G}_k^2\pare{t} = \frac{1}{\delta}\ \frac{\d}{\d t} \ \bra{ \widetilde{F}_k^2 \pare{t}}, 
\end{align*}
while using the identity \eqref{eq:H_in_fz_di_M} we argue that
\begin{equation}\label{eq:def_FkGk}
\begin{aligned}
\norm{H}_{H^k\pare{t}} & \leqslant C \pare{\norm{M\pare{t}}_{H^k} + \norm{\cG_F\pare{t}}_{H^k}}, \\
& \leqslant C \pare{\widetilde{F}_k\pare{t} + \norm{\cG_F }_{L^\infty \pare{ \bra{0, T};  H^k}}} && \overset{\text{def}}{ =} \frac{1}{2}\ F_k\pare{t}, \\
\norm{ \nabla H}_{H^k\pare{t}} & \leqslant C \pare{\norm{\nabla M\pare{t}}_{H^k} + \norm{ \nabla \cG_F\pare{t}}_{H^k}}, \\
& \leqslant C \pare{\widetilde{G}_k\pare{t} + \norm{\nabla\cG_F }_{L^\infty \pare{ \bra{0, T};  H^k}}} && \overset{\text{def}}{ =} \frac{1}{2}\ G_k\pare{t}.
\end{aligned}
\end{equation}

\noindent By definition $ F_k\geqslant \widetilde{F}_k $ and $ G_k\geqslant \widetilde{G}_k $, hence $ U=\pare{u, \omega, M, H} $ satisfies the Hypothesis \ref{hyp:induction}  for $ k $ when $ F_k $ and $ G_k $ are defined as in \eqref{eq:def_FkGk}, proving hence the induction and concluding the proof. 
\hfill $ \Box $

\appendix

\section{Technical estimates}\label{appendix_technical_estimates}

\subsection{Proof of Lemma \ref{lem:L2_energy_estimates}:} \label{pf:L2_bounds}
Let us multiply the equation of  $ u $ in \eqref{eq:Rosensweig2D} for $ u $ and let us integrate in $ \mathbb{R}^2 $, obtaining
\begin{equation}\label{eq:L2est0}
\frac{\rho_0}{2} \frac{\d}{\d t} \norm{u}^2_{L^2} + \pare{\eta + \zeta} \norm{\nabla u}^2_{L^2} = \mu_0\ps{M \cdot\nabla H}{u}_{L^2} + 2\zeta \ps{\pare{\begin{array}{c}
\partial_2 \omega\\ -\partial_1\omega
\end{array}}}{u}_{L^2}, 
\end{equation}
Since $ \curl \ H =0 $ we can assert that there exists a scalar function $ \phi_H $ such that $ H=\nabla \phi_H $. Hence we can deduce the identity $ \mu_0\ps{M \cdot\nabla H}{u}_{L^2}= -\mu_0 \ps{u\cdot\nabla M}{H}_{L^2} $. Multiplying the equation describing the evolution of $ M $ in \eqref{eq:Rosensweig2D} for $ H $ and integrating in space  we deduce that
\begin{equation*}
\ps{\partial_t M}{H}_{L^2} + \ps{u\cdot\nabla M}{H}_{L^2} - \sigma\ps{\Delta M}{H}_{L^2} = \ps{\pare{
\begin{array}{c}
-M_2\\ M_1
\end{array}
}\omega}{H}_{L^2} - \frac{1}{\tau} \ps{M-\chi_0 H}{H}_{L^2}, 
\end{equation*}
which combined with \eqref{eq:L2est0} implies the following equality
\begin{multline}\label{eq:L2est1}
\frac{\rho_0}{2} \frac{\d}{\d t} \norm{u}^2_{L^2} + \pare{\eta + \zeta} \norm{\nabla u}^2_{L^2}
 =   \mu_0 \ps{\partial_t M}{H}_{L^2}  - \mu_0 \sigma\ps{\Delta M}{H}_{L^2} \\
 - \mu_0\ps{\pare{
\begin{array}{c}
-M_2\\ M_1
\end{array}
}\omega}{H}_{L^2} + \frac{\mu_0}{\tau} \ps{M-\chi_0 H}{H}_{L^2}
+ 2\zeta \ps{\pare{\begin{array}{c}
\partial_2 \omega\\ -\partial_1\omega
\end{array}}}{u}_{L^2}
.
\end{multline}
Multiplying the fourth equation of \eqref{eq:Rosensweig2D} for $ \phi_H $, integrating by parts, considering that $ H=\nabla \phi_H $ and integrating we obtain the equation
\begin{equation}\label{eq:psMH}
\begin{aligned}
-\ps{M}{H}_{L^2}& = \norm{H}_{L^2}^2 + \int \div\nabla \Delta^{-1}F \cdot \phi_H \ \dx, \\
 & = \norm{H}_{L^2}^2 - \ps{\cG_F}{H}_{L^2},\end{aligned} 
\end{equation}
from which we deduce 
\begin{align*}
\frac{\mu_0}{\tau} \ps{M-\chi_0 H}{H}_{L^2} & = -  \frac{\mu_0}{\tau} \pare{1+\chi_0}\norm{H}_{L^2}^2 + \frac{\mu_0}{\tau} \ps{\cG_F}{H}_{L^2}, \\
& \leqslant -  \frac{\mu_0}{\tau} \pare{\frac{3}{4}+\chi_0}\norm{H}_{L^2}^2 + \frac{C}{\tau}\norm{\cG_F}_{L^2}^2. 
\end{align*}

While differentiating in time the magnetostating equation, multiplying for $ \phi_H $ and integrating in space we deduce
\begin{align*}
\mu_0 \ps{\partial_t M}{H}_{L^2} & = -\frac{\mu_0}{2} \  \frac{\d}{\d t} \norm{H}^2_{L^2} + \mu_0 \ps{\partial_t \cG_F}{H}_{L^2}, \\
& \leqslant -\frac{\mu_0}{2} \  \frac{\d}{\d t} \norm{H}^2_{L^2} +  \frac{\mu_0}{4\tau} \norm{H}_{L^2}^2 + C\mu_0 \tau \norm{\partial_t \cG_F}^2. 
\end{align*}

Taking in consideration the magnetostatic equation $ \div\pare{M+H}=F $, recalling that $ \Delta M = \div\nabla M $,  integrating by parts,  using the identity \eqref{eq:H_in_fz_di_M} and a Young inequality we derive 
\begin{align*}
- \mu_0 \sigma\ps{\Delta M}{H}_{L^2} & = -\mu_0 \sigma\norm{\div\ M}_{L^2}^2 + \mu_0 \sigma\int \div \ M \ F \dx, \\
& \leqslant -\frac{\mu_0 \sigma}{2} \norm{\div\ M}_{L^2}^2 + C \norm{F}_{L^2}^2
. 
\end{align*} 

Next we consider the identity
\begin{equation*}
2\zeta \ps{\pare{\begin{array}{c}
\partial_2 \omega\\ -\partial_1\omega
\end{array}}}{u}_{L^2} = 2\zeta \ps{\curl \ u}{\omega}_{L^2},
\end{equation*}
which holds true with a simple integration by parts,  these considerations transform \eqref{eq:L2est1} in
\begin{multline}\label{eq:L2est2}
\frac{1}{2}\ \frac{\d}{\d t}\pare{{\rho_0}\norm{u}^2_{L^2} + {\mu_0} \   \norm{H}^2_{L^2}} + \pare{\eta + \zeta} \norm{\nabla u}^2_{L^2} + \frac{\mu_0 \sigma}{2} \norm{\div\ M}_{L^2}^2 +\frac{\mu_0}{\tau} \pare{\frac{1}{2}+\chi_0}\norm{H}_{L^2}^2
\\
 \leqslant  C\mu_0 \tau \norm{\partial_t \cG_F} ^2_{L^2}
 - \mu_0\ps{\pare{
\begin{array}{c}
-M_2\\ M_1
\end{array}
}\omega}{H}_{L^2}
+ 2\zeta \ps{\curl \ u}{\omega}_{L^2} +\frac{C}{\tau}\norm{\cG_F}_{L^2}^2 + C \norm{F}_{L^2}^2
.
\end{multline}

Let us now perform an $ L^2 $ energy estimate on the equation for $ \omega $ in \eqref{eq:Rosensweig2D}:

\begin{equation}\label{eq:L2est3}
\frac{\rho_0 \kappa}{2} \ \frac{\d}{\d t} \norm{\omega}_{L^2}^2 + \eta'\norm{\nabla\omega}^2_{L^2} = \mu_0 \ps{M\times H}{\omega}_{L^2}  + 2\zeta \ps{\curl \ u -2\omega}{\omega}_{L^2}. 
\end{equation}
The following algebraic identity is immediate
\begin{align*}
\mu_0 \ps{M\times H}{\omega}_{L^2} & = \mu_0\ps{\pare{
\begin{array}{c}
-M_2\\ M_1
\end{array}
}\omega}{H}_{L^2},
\end{align*}
whence we can add \eqref{eq:L2est2} and \eqref{eq:L2est3} to deduce the following inequality
\begin{multline}\label{eq:L2est4}
\frac{1}{2}\ \frac{\d}{\d t}\pare{{\rho_0}\norm{u}^2_{L^2} + {\mu_0} \   \norm{H}^2_{L^2} + \rho_0 \kappa  \norm{\omega}_{L^2}^2 } \\
 + \pare{\eta + \zeta} \norm{\nabla u}^2_{L^2} + \eta'\norm{\nabla\omega}^2_{L^2} + \frac{\mu_0 \sigma}{2} \norm{\div\ M}_{L^2}^2 +\frac{\mu_0}{\tau} \pare{\frac{1}{2}+\chi_0}\norm{H}_{L^2}^2
\\
\leqslant C\mu_0 \tau \norm{\partial_t \cG_F} ^2_{L^2}
 +2\zeta \ps{\curl \ u}{\omega}_{L^2} - 4\zeta\norm{\omega}_{L^2}^2
 +\frac{C}{\tau}\norm{\cG_F}_{L^2}^2 + C \norm{F}_{L^2}^2. 
\end{multline}
Since $ \norm{\nabla u}_{L^2}^2 = \sum_{i,j=1}^2\int \av{\partial_i u_j}^2 \dx $ an application of Cauchy-Schwartz inequality shows that
\begin{equation*}
\zeta \norm{\nabla u}_{L^2}^2  + 4 \zeta\norm{\omega}_{L^2}^2 -\zeta \ps{\curl \ u}{2 \omega}_{L^2} \geqslant 0,
\end{equation*}
whence we can improve the bound in \eqref{eq:L2est4} with
\begin{multline}
\label{eq:L2est5}
\frac{1}{2}\ \frac{\d}{\d t}\pare{{\rho_0}\norm{u}^2_{L^2} + {\mu_0} \   \norm{H}^2_{L^2} + \rho_0 \kappa  \norm{\omega}_{L^2}^2 } \\
 + \eta \norm{\nabla u}^2_{L^2} + \eta'\norm{\nabla\omega}^2_{L^2} + \frac{\mu_0 \sigma}{2} \norm{\div\ M}_{L^2}^2 +\frac{\mu_0}{\tau} \pare{\frac{1}{2}+\chi_0}\norm{H}_{L^2}^2
\\
\leqslant  C\mu_0 \tau \norm{\partial_t \cG_F}_{L^2}^2 
 +\frac{C}{\tau}\norm{\cG_F}_{L^2}^2 + C \norm{F}_{L^2}^2.
\end{multline}

At last we perform an $ L^2 $ energy estimate on the equation for $ M $ and we deduce
\begin{equation*}
\frac{1}{2} \ \frac{\d}{\d t} \norm{M}_{L^2}^2 + \sigma \norm{\nabla M}_{L^2}^2 = -\frac{1}{\tau}\norm{M}_{L^2}^2 + \frac{\chi_0}{\tau} \ps{H}{M}_{L^2}, 
\end{equation*}
whence  with the identity \eqref{eq:psMH}
\begin{equation}
\label{eq:L2est6}
\begin{aligned}
\frac{1}{2} \ \frac{\d}{\d t} \norm{M}_{L^2}^2 + \sigma \norm{\nabla M}_{L^2}^2 +\frac{1}{\tau}\norm{M}_{L^2}^2 & = - 
\frac{\chi_0}{\tau} \pare{  \norm{H}_{L^2}^2 + \ps{\cG_F}{H}_{L^2}}, \\
& \leqslant -\frac{\chi_0}{2\tau}\norm{H}_{L^2}^2 + C\norm{\cG_F}_{L^2}^2,
\end{aligned}
\end{equation}
whence adding \eqref{eq:L2est5} and \eqref{eq:L2est6} we deduce the final inequality
\begin{multline}
\label{eq:L2est7}
\frac{1}{2}\ \frac{\d}{\d t}\pare{{\rho_0}\norm{u}^2_{L^2} + {\mu_0} \   \norm{H}^2_{L^2} + \rho_0 \kappa  \norm{\omega}_{L^2}^2 + \norm{M}_{L^2}^2} \\
 + \eta \norm{\nabla u}^2_{L^2} + \eta'\norm{\nabla\omega}^2_{L^2} + \sigma \norm{\nabla M}_{L^2}^2 + \frac{\mu_0 \sigma}{2} \norm{\div\ M}_{L^2}^2 +\frac{1}{ \tau} \pare{ \frac{\mu_0}{2} + \chi_0\pare{\mu_0+\frac{1}{2}}}\norm{H}_{L^2}^2 +\frac{1}{\tau}\norm{M}_{L^2}^2
\\
\leqslant  C\mu_0 \tau \norm{\partial_t \cG_F} _{L^2}^2
 +\frac{1}{\tau}\norm{\cG_F}_{L^2}^2 +  \norm{F}_{L^2}^2.
\end{multline}

We can hence reformulate  \eqref{eq:L2est7} with the quantities defined in \eqref{eq:ctilde}--\eqref{eq:cFtau}, with such and an integration in time we deduce the integral inequality
\begin{equation*}
\frac{1}{2}\cE\pare{t} + \tilde{c}\int_0^t \cE_d  \pare{t'}\d t' \leqslant \frac{1}{2} \cE\pare{0} + C \int_0^t f_\tau \pare{t'}\d t',
\end{equation*}
which concludes the proof. 
\hfill $ \Box $

\subsection{Proof of Lemma \ref{lem:H12_energy_estimates}: }\label{appendix:H12_energy_estimates}
The proof of Lemma \ref{lem:H12_energy_estimates} consists in performing some $ \Hud $ energy estimates on the system \eqref{eq:Rosensweig2D} and to check that, as long as the estimates of Lemma \ref{lem:L2_energy_estimates} hold as well, we can obtain a global control for the $ \Hud $ regularity of the solutions of \eqref{eq:Rosensweig2D}. Let us hence multiplying the equation describing the evolution of $ u $ for $ \Lambda u $ and integrating in space we deduce the following energy equality
\begin{multline*}
\frac{\rho_0}{2}\ \frac{\d}{\d t} \norm{\Lambda^{\sfrac{1}{2}} u}_{L^2}^2 + \pare{\eta + \zeta}\norm{\Lambda^{\sfrac{1}{2}}\nabla u}_{L^2}
 \leqslant \av{\ps{\Lambda^{\sfrac{1}{2}}\pare{u\cdot\nabla u}}{\Lambda^{\sfrac{1}{2}}u}_{L^2}}\\
  + \mu_0 \av{\ps{\Lambda^{\sfrac{1}{2}}\pare{M\cdot \nabla H}}{\Lambda^{\sfrac{1}{2}} u}_{L^2}}
 + 2\zeta\av{\ps{\Lambda^{\sfrac{1}{2}}\pare{\begin{array}{c}
\partial_2 \omega\\ -\partial_1\omega
\end{array}}}{\Lambda^{\sfrac{1}{2}}u}_{L^2}}, 
\end{multline*}
hence using repeatedly integration by parts and Lemma \ref{lem:prod_rules_Sobolev_2D} 
\begin{align*}
\av{\ps{\Lambda^{\sfrac{1}{2}}\pare{u\cdot\nabla u}}{\Lambda^{\sfrac{1}{2}}u}_{L^2}} &  \leqslant \norm{u\cdot \nabla u}_{L^2} \norm{\nabla u}_{L^2}, \\
& \leqslant \norm{\nabla u}_{L^2}\norm{ \Lambda^{\sfrac{1}{2}} u}_{L^2}\norm{\Lambda^{\sfrac{1}{2}}\nabla u}_{L^2}, \\
& \leqslant \frac{\eta+\zeta}{2} \norm{\Lambda^{\sfrac{1}{2}}\nabla u}_{L^2}^2 + C \norm{\nabla u}_{L^2}^2\norm{ \Lambda^{\sfrac{1}{2}} u}_{L^2}^2, \\
\av{\ps{\Lambda^{\sfrac{1}{2}}\pare{\begin{array}{c}
\partial_2 \omega\\ -\partial_1\omega
\end{array}}}{\Lambda^{\sfrac{1}{2}} u}_{L^2}} & \leqslant \norm{\nabla \omega}_{L^2}\norm{\nabla u}_{L^2}, 
\end{align*}
and hence considering as well the estimate given in Lemma \ref{lem:bound_Lorentz_force} we deduced the following inequality applying repeatedly the convexity inequality $ ab\leqslant \frac{a^2}{2}+\frac{b^2}{2} $:
\begin{multline}\label{eq:H12u}
\frac{\rho_0}{2}\ \frac{\d}{\d t} \norm{\Lambda^{\sfrac{1}{2}} u}_{L^2}^2 + \frac{\eta+\zeta}{2} \norm{\Lambda^{\sfrac{1}{2}}\nabla u}_{L^2} \leqslant
\frac{\sigma}{8} \norm{\Lambda^{\sfrac{1}{2}}\nabla M}_{L^2}^2 +\\
C \norm{\nabla u}_{L^2}^2\norm{\Lambda^{\sfrac{1}{2}} u}_{L^2}^2 +
\norm{\nabla \omega}_{L^2}\norm{\nabla u}_{L^2} 
+ C \norm{\nabla u}^2_{L^2} \norm{\Lambda^{\sfrac{1}{2}} M}^2_{L^2} + \norm{\Lambda^{\sfrac{1}{2}}\nabla \cG_F}_{L^2}^2. 
\end{multline}

Multiplying the equation describing the evolution of $ \omega $ in \eqref{eq:Rosensweig2D} for $ \Lambda \omega $ and integrating in space we deduce instead
\begin{multline*}
\frac{\rho_0\kappa}{2} \ \frac{\d}{\d t}\norm{\Lambda^{\sfrac{1}{2}}\omega}_{L^2}^2 +\eta' \norm{\Lambda^{\sfrac{1}{2}}\nabla \omega}_{L^2}^2 + 4\zeta \norm{\Lambda^{\sfrac{1}{2}} \omega}_{L^2}^2 \\
 \leqslant \av{\ps{\Lambda^{\sfrac{1}{2}}\pare{M\times H}}{\Lambda^{\sfrac{1}{2}}\omega}_{L^2}} +2\zeta \av{\ps{\Lambda^{\sfrac{1}{2}}\curl\ u}{\Lambda^{\sfrac{1}{2}}\omega}_{L^2}} , 
\end{multline*}
and
\begin{align*}
\av{\ps{\Lambda^{\sfrac{1}{2}}\pare{M\times H}}{\Lambda^{\sfrac{1}{2}}\omega}_{L^2}} & \leqslant \norm{\Lambda^{\sfrac{1}{2}}\nabla\omega}_{L^2}\norm{\Lambda^{-\sfrac{1}{2}}\pare{M\times H}}_{L^2}, \\
& \leqslant C \norm{\Lambda^{\sfrac{1}{2}}\nabla\omega}_{L^2} \norm{\Lambda^{\sfrac{1}{2}} M}_{L^2}\norm{H}_{L^2}, \\
& \leqslant \frac{\eta'}{4}\norm{\Lambda^{\sfrac{1}{2}}\nabla\omega}_{L^2}^2 + C \norm{H}_{L^2}^2 \norm{\Lambda^{\sfrac{1}{2}} M}_{L^2}^2, \\
\av{\ps{\Lambda^{\sfrac{1}{2}}\curl\ u}{\Lambda^{\sfrac{1}{2}}\omega}_{L^2}} & \leqslant \norm{\nabla \omega}_{L^2}\norm{\nabla u}_{L^2} . 
\end{align*}
Whence we deduced the inequality
\begin{equation}\label{eq:H12omega}
\frac{\rho_0\kappa}{2} \ \frac{\d}{\d t}\norm{\Lambda^{\sfrac{1}{2}}\omega}_{L^2}^2 +\frac{\eta'}{2} \norm{\Lambda^{\sfrac{1}{2}}\nabla \omega}_{L^2}^2 + 4\zeta \norm{\Lambda^{\sfrac{1}{2}}\omega}_{L^2}^2 \leqslant C \norm{H}_{L^2}^2 \norm{\Lambda^{\sfrac{1}{2}} M}_{L^2}^2 + \norm{\nabla \omega}_{L^2}\norm{\nabla u}_{L^2}. 
\end{equation}

\noindent We perform the same procedure onto the equation for $ M $ deducing hence
\begin{multline*}
\frac{1}{2} \ \frac{\d}{\d t} \norm{\Lambda^{\sfrac{1}{2}} M}_{L^2}^2 + \sigma\norm{\Lambda^{\sfrac{1}{2}} \nabla M}_{L^2}^2 \leqslant \av{\ps{\Lambda^{\sfrac{1}{2}}\pare{u\cdot\nabla M}}{\Lambda^{\sfrac{1}{2}} M}_{L^2}}\\
 + \av{\ps{\Lambda^{\sfrac{1}{2}}\pare{ \pare{
\begin{array}{c}
-M_2\\ M_1
\end{array}
}\omega}}{\Lambda^{\sfrac{1}{2}} M}_{L^2}} -\frac{1}{\tau}\ps{\Lambda^{\sfrac{1}{2}}M-\chi_0\Lambda^{\sfrac{1}{2}} H}{\Lambda^{\sfrac{1}{2}}M}_{L^2}.
\end{multline*}

\noindent Straightforward calculations prove the following bounds
\begin{align*}
\av{\ps{\Lambda^{\sfrac{1}{2}}\pare{u\cdot\nabla M}}{\Lambda^{\sfrac{1}{2}} M}_{L^2}} & \leqslant \frac{\sigma}{8}\norm{\Lambda^{\sfrac{1}{2}}\nabla M}_{L^2}^2 + \norm{\nabla M}_{L^2}^2 \norm{\Lambda^{\sfrac{1}{2}} u}_{L^2}^2, \\
\av{\ps{\Lambda^{\sfrac{1}{2}} \pare{ \pare{
\begin{array}{c}
-M_2\\ M_1
\end{array}
}\omega}}{\Lambda^{\sfrac{1}{2}} M}_{L^2}} & \leqslant \norm{M}_{L^2} \norm{\Lambda^{\sfrac{1}{2}}\omega}_{L^2} \norm{\Lambda^{\sfrac{1}{2}}\nabla M}_{L^2}, \\
& \leqslant \frac{\sigma}{8} \norm{\Lambda^{\sfrac{1}{2}} \nabla M}_{L^2}^2 + C \norm{M}_{L^2}^2\norm{\Lambda^{\sfrac{1}{2}} \omega}_{L^2}^2.
\end{align*}

\noindent For the last term it suffice to notice that
\begin{equation*}
-\frac{1}{\tau}\ps{\Lambda^{\sfrac{1}{2}}M-\chi_0\Lambda^{\sfrac{1}{2}} H}{\Lambda^{\sfrac{1}{2}}M}_{L^2} = -\frac{1}{\tau}\pare{ \norm{\Lambda^{\sfrac{1}{2}}M}_{L^2}^2 + \chi_0\ps{\Lambda^{\sfrac{1}{2}}H}{\Lambda^{\sfrac{1}{2}}M}_{L^2}}, 
\end{equation*}
and moreover since $ H=-\cQ M +\cG_F $ as it was explained in \eqref{eq:H_in_fz_di_M} we deduce that
\begin{equation*}
\ps{\Lambda^{\sfrac{1}{2}} H}{\Lambda^{\sfrac{1}{2}} M}_{L^2}= -\norm{\Lambda^{\sfrac{1}{2}}\cQ M}_{L^2}^2 + \ps{\Lambda^{\sfrac{1}{2}}\cG_F}{\Lambda^{\sfrac{1}{2}}M}_{L^2}. 
\end{equation*}
Since $ \cG_F = \nabla\Delta^{-1}F $ we immediately deduce that $$ \ps{\Lambda^{\sfrac{1}{2}}\cG_F}{\Lambda^{\sfrac{1}{2}}M}_{L^2}= \ps{\Lambda^{\sfrac{1}{2}}\cG_F}{\Lambda^{\sfrac{1}{2}}\cQ M}_{L^2}\leqslant  \frac{1}{2}\norm{\Lambda^{\sfrac{1}{2}}\cQ M}_{L^2}^2 + C \norm{\Lambda^{\sfrac{1}{2}}\cG_F}_{L^2}^2 ,  $$ whence we conclude that
\begin{equation*}
-\frac{1}{\tau}\ps{\Lambda^{\sfrac{1}{2}}M-\chi_0\Lambda^{\sfrac{1}{2}} H}{\Lambda^{\sfrac{1}{2}}M}_{L^2}\leqslant -\frac{1}{\tau}\pare{  \norm{\Lambda^{\sfrac{1}{2}} M}_{L^2}^2 + \frac{\chi_0}{2}\norm{\Lambda^{\sfrac{1}{2}}\cQ M}_{L^2}^2} + \frac{C}{\tau}\norm{\Lambda^{\sfrac{1}{2}} \cG_F}_{L^2}^2. 
\end{equation*}

\noindent We recover hence the final $ \Hud $ energy inequality for $ M $: 
\begin{multline}\label{eq:H12M}
\frac{1}{2} \ \frac{\d}{\d t} \norm{\Lambda^{\sfrac{1}{2}} M}_{L^2}^2 + \frac{\sigma}{2}\norm{\Lambda^{\sfrac{1}{2}} \nabla M}_{L^2}^2 + \frac{1}{\tau}\pare{\norm{\Lambda^{\sfrac{1}{2}}M}_{L^2}^2 + \frac{\chi_0}{2}\norm{\Lambda^{\sfrac{1}{2}}\cQ M}_{L^2}^2}\\
 \leqslant  C  \pare{\norm{\nabla M}_{L^2}^2 \norm{\Lambda^{\sfrac{1}{2}}u}_{L^2}^2 +\norm{M}_{L^2}^2\norm{\Lambda^{\sfrac{1}{2}}\omega}_{L^2}^2} + \frac{C}{\tau}\norm{\Lambda^{\sfrac{1}{2}} \cG_F}_{L^2}^2. 
\end{multline}

Adding the inequalities \eqref{eq:H12u}, \eqref{eq:H12omega} and \eqref{eq:H12M} we recover the differential inequality
\begin{multline}\label{eq:H12est_almostfinal}
\frac{1}{2}\ \frac{\d}{\d t} \pare{ \rho_0 \norm{\Lambda^{\sfrac{1}{2}} u}_{L^2}^2 + \rho_0\kappa \norm{\Lambda^{\sfrac{1}{2}}\omega}_{L^2}^2 + \norm{\Lambda^{\sfrac{1}{2}} M}_{L^2}^2}\\
 + \left[ \frac{\eta+\zeta}{2} \norm{\Lambda^{\sfrac{1}{2}}\nabla u}_{L^2} +\frac{\eta'}{2} \norm{\Lambda^{\sfrac{1}{2}} \nabla \omega}_{L^2}^2 + 4\zeta \norm{\Lambda^{\sfrac{1}{2}} \omega}_{L^2}^2 \right. \\
  + \left.  \frac{\sigma}{2}\norm{\Lambda^{\sfrac{1}{2}} \nabla M}_{L^2}^2 + \frac{1}{\tau}\pare{\norm{\Lambda^{\sfrac{1}{2}} M}_{L^2}^2 + \frac{\chi_0}{2}\norm{\Lambda^{\sfrac{1}{2}} \cQ M}_{L^2}^2} \right]\\
 \leqslant C \left[ \norm{\nabla u}_{L^2}^2  \norm{\Lambda^{\sfrac{1}{2}} u}_{L^2}^2 
+  \norm{\nabla u}^2_{L^2} \norm{\Lambda^{\sfrac{1}{2}} M}^2_{L^2}  \right.  +
\left. \norm{H}_{L^2}^2 \norm{\Lambda^{\sfrac{1}{2}} M}_{L^2}^2 + \norm{\nabla M}_{L^2}^2 \norm{\Lambda^{\sfrac{1}{2}} u}_{L^2}^2 +\norm{M}_{L^2}^2\norm{\Lambda^{\sfrac{1}{2}} \omega}_{L^2}^2\right] \\
 + \bra{ \norm{\Lambda^{\sfrac{1}{2}} \nabla \cG_F}_{L^2}^2 + \frac{C}{\tau}\norm{\Lambda^{\sfrac{1}{2}} \cG_F}_{L^2}^2+
\norm{\nabla \omega}_{L^2}\norm{\nabla u}_{L^2}} . 
\end{multline}

\noindent With the quantities defined in \eqref{eq:def_c} we can deduce from  equation \eqref{eq:H12est_almostfinal} the following differential inequality
\begin{equation*}
\frac{1}{2} \ \frac{\d}{\d t} \ \cF\pare{t} + c \ \cF_d\pare{t} \leqslant C \cE_d\pare{t} \ \cF\pare{t} + \Phi_\tau \pare{t} +C \cE_d\pare{t} . 
\end{equation*}
where $ \cE_d $ is defined in \eqref{eq:def_Ed} which is an $ L^1 $ function in time thanks to the bounds provided in Lemma \ref{lem:L2_energy_estimates}. An application of Gronwall inequality allows hence to deduce the inequality
\begin{multline}\label{eq:H12_energy_eq_almost_final}
\cF\pare{t} + 2c \int_0^t \exp\set{\int _{t'}^t \cE_d\pare{t''}\d t''} \cF_d \pare{t'} \d t' \leqslant C \  \cF\pare{0} \exp\set{C \int_0^t\cE_d\pare{t'}\d t'} \\
+C \ \int_0^t \exp\set{ C \int _{t'}^t \cE_{d}\pare{t''} \d t''} \bra{ \Big. \Phi_\tau \pare{t'} + \cE_d\pare{t'}}\d t'. 
\end{multline} 

Whence considering the $ L^2 $ energy bound \eqref{eq:L2_energy_bound} we can argue that, for each $ 0\leqslant t'\leqslant t\leqslant T $
\begin{equation*}
1 \leqslant \int _{t'}^t \cE_d\pare{t''}\d t'' \leqslant \frac{1}{\tilde{c}} \ \Psi\pare{U_0, F, \cG_F}, 
\end{equation*}
which in turn implies that
\begin{multline*}
\cF \pare{t} + 2c \int_0^t  \cF_d \pare{t'} \d t'\\ 
 \leqslant C \  \cF\pare{0}\exp\set{\frac{C}{\tilde{c}} \ \Psi\pare{U_0, F, \cG_F}} + C \ \exp\set{\frac{C}{\tilde{c}} \ \Psi\pare{U_0, F, \cG_F}} \norm{\cG_F}_{L^2\pare{\bra{0, T}; H^{\frac{3}{2}}}} \\
 + \frac{C}{\tilde{c}} \ \exp\set{\frac{C}{\tilde{c}} \ \Psi\pare{U_0, F, \cG_F}} \ \Psi\pare{U_0, F, \cG_F}. 
\end{multline*}
In order to deduce \eqref{eq:H_in_Eud} it suffice to apply Lemma \ref{lem:reg_H} and to consider that $ \cG_F $ was considered to be in $ W^{1, \infty}_{\loc}\pare{\bR_+; H^{\sfrac{3}{2}}} $. 
\hfill $ \Box $

\subsection{Proof of the uniqueness statement in Proposition \ref{prop:existence_global_weak_H12_solutions}:} \label{sec:uniquess_weak}
We want to prove here that the solutions constructed in Section \ref{sec:weak_solutions} are unique in the  energy space $ L^\infty_{\loc}\pare{ \bR_+; H^{\sfrac{1}{2}}}\cap L^2_{\loc} \pare{ \bR_+;  H^{\sfrac{3}{2}}} $. In order to do so let us hence denote  $ V=\pare{u, \omega, M} $
 as above and let us write the system \eqref{eq:Rosensweig2D_limit} in the compact form
 \begin{equation}\label{eq:Rosensweig2D_limit_compact}
 \left\lbrace
 \begin{aligned}
 & \partial_t V -\mathcal{L} \ V = B_2\pare{V, V} + B_1\pare{V, V} + L \ V + f_{\text{ext}}, \\
 & \left. V\right|_{t=0} =V_0, 
 \end{aligned}
 \right.
 \end{equation}
 where respectively
 \begin{align*}
 \cL \ V  = \pare{\begin{array}{c}
 \pare{\eta +\zeta} \Delta u \\ \eta' \Delta \omega \\ \sigma \Delta M
 \end{array} }, && 
 L\ V & = \pare{
\begin{array}{c}
\mu_0 \cP \pare{ M\cdot \nabla \cG_F} -2\zeta \ \cP \nabla^\perp \omega\\
\mu_0 \ M\times \cG_F + 2\zeta \pare{ \curl \ \cP u -2\omega}\\
-\frac{1}{\tau}\pare{1+\chi_0 \cQ }M
\end{array} 
}, 
 &&
f_{\text{ext}} = \pare{
\begin{array}{c}
0\\ 0 \\-\frac{1}{\tau}\cG_F
\end{array}
} , 
 \end{align*}
 while the bilinear interactions $ B_2 $ and $ B_1 $ are defined as
 \begin{align*}
 B_2\pare{V, V} & = \pare{ \begin{array}{c} 
- \rho_0\cP  \pare{\cP u \cdot \nabla \cP u}   -\mu_0 \cP \pare{M\cdot \nabla \cQ M} \\
-\rho_0\kappa \cP u\cdot\nabla \omega \\
-\cP u\cdot \nabla M
  \end{array}
 }, &
 B_1\pare{V, V} & = \pare{
\begin{array}{c}
0 \\
- \mu_0 \ M\times \cQ M \\
M^\perp \omega
\end{array} 
 }. 
\end{align*}
And let $ V_i, \ i=1, 2 $ be a global weak solutions of the following Cauchy problem
\begin{equation*}
 \left\lbrace
 \begin{aligned}
 & \partial_t V_i -\mathcal{L} \ V_i = B_2\pare{V_i, V_i} + B_1\pare{V_i, V_i} + L \ V_i + f_{\text{ext}}, \\
 & \left. V_i\right|_{t=0} =V_{i, 0}, 
 \end{aligned}
 \right.
 \end{equation*}
 where $ V_{i, 0}, \ i=1,2  $ belong to $ H^{\sfrac{1}{2}} $. As explained above 
 \begin{align*}
 V_i\in L^\infty_{\loc}\pare{\bR_+; \Hud}, 
 &&
 \nabla V_i\in L^2_{\loc}\pare{\bR_+; \Hud}. 
 \end{align*}
 Let us hence define $ \delta V = V_1-V_2 $ and $ \delta V_0=V_{1, 0}-V_{2, 0} $, then $ \delta V $ solves weakly
 \begin{equation}\label{eq:Rosensweig2D_difference_compact_form}
 \left\lbrace
 \begin{aligned}
 & \partial_t \delta V -\mathcal{L} \ \delta V = B_2\pare{V_1, \delta V} + B_2 \pare{\delta V, V_2} + B_1\pare{V_1, \delta V} + B_1 \pare{\delta V, V_2} + L \ \delta V , \\
 & \left. \delta V\right|_{t=0} =\delta V_0.
 \end{aligned}
 \right.
 \end{equation}
 
 The method we will adopt in order to prove the uniqueness of solutions of \eqref{eq:Rosensweig2D_limit_compact} in $ L^\infty_{\loc}\pare{\bR_+; H^{\sfrac{1}{2}}} \cap L^2_{\loc}\pare{\bR_+; H^{\sfrac{3}{2}}} $ is rather standard and it develops in the following way
 \begin{enumerate}
 \item[$ \diamond $] We perform some $ L^2 $ energy estimates on the system \eqref{eq:Rosensweig2D_difference_compact_form} in order to deduce a bound of the following form
 \begin{equation*}
 \norm{\delta V}_{L^\infty\pare{\bra{0, t}; L^2} }^2  + \norm{\nabla \delta V}_{L^2\pare{\bra{0, t};  L^2} }^2  \leqslant \norm{\delta V_0}_{L^2} f\pare{t},
 \end{equation*}
 for any $ t>0 $ and some $ f\in L^\infty_{\loc}\pare{\bR_+} $. 
 
  \item[$ \diamond $] We perform next some $ \Lambda^{\sfrac{1}{2}} L^2 $ energy estimates always on the system \eqref{eq:Rosensweig2D_difference_compact_form} in order to deduce the energy inequality
  \begin{equation*}
   \norm{\delta V}_{L^\infty\pare{\bra{0, t}; \Lambda^{\sfrac{1}{2}} L^2} }^2 + \norm{\nabla \delta V}_{L^2\pare{\bra{0, t}; \Lambda^{\sfrac{1}{2}} L^2} }^2  \leqslant \norm{\delta V_0}_{\Lambda^{\sfrac{1}{2}} L^2} g\pare{t},
  \end{equation*}
  for any $ t>0 $ and some $ g\in L^\infty_{\loc}\pare{\bR_+} $.
  
  \item[$ \diamond $] Interpolating the two inequalities here above and setting $ \delta V = 0 $ in $ H^{\sfrac{1}{2}} $ we obtain hence that $ \delta V $ has to be identically nil in the space $ L^\infty_{\loc}\pare{\bR_+; H^{\sfrac{1}{2}}} \cap L^2_{\loc}\pare{\bR_+; H^{\sfrac{3}{2}}} $. 
 \end{enumerate}
 
 From now on we use the notation $ \norm{\cdot}_{s} = \norm{\Lambda^s \cdot}_{L^2}, \ \ps{\Lambda^s\cdot}{\Lambda^s\cdot}_{L^2}=\ps{\cdot}{\cdot}_{s} $ for any $ s\in \bR $.

 \subsubsection{Step 1: The $  L^2 $ energy bound :}
 the bound we provide in this first step are relatively simple, hence we will sometimes omit to specify in detail every step required in order to prove them.\\
 
 We multiply the equation \eqref{eq:Rosensweig2D_difference_compact_form} for $ \delta V $ and we integrate in space in order to deduce the differential inequality
 \begin{multline}\label{eq:ineq_unicity_-3}
 \frac{1}{2}\ \frac{\d}{\d t} \norm{\delta V}_{L^2}^2 + c\norm{\nabla\delta V}_{L^2}^2 \leqslant
 \av{\ps{ B_2\pare{V_1, \delta V}}{\delta V}_{L^2}}
 + \av{\ps{ B_2\pare{\delta V, V_2}}{\delta V}_{L^2}}\\
 + \av{\ps{ B_1\pare{V_1, \delta V}}{\delta V}_{L^2}}
 + \av{\ps{ B_1\pare{\delta V, V_2}}{\delta V}_{L^2}}
 +\av{\ps{L \ \delta V}{\delta V}_{L^2}}. 
\end{multline} 

The following estimates are immediate to deduce
\begin{equation}\label{eq:ineq_unicity_-2}
\begin{aligned}
\av{\ps{ B_2\pare{V_1, \delta V}}{\delta V}_{L^2}} & \leqslant  \norm{V_1}_{L^\infty} \norm{ \delta V}_{L^2}\norm{\nabla \delta V}_{L^2} , \\
& \leqslant \alpha \norm{\nabla \delta V}_{L^2}^2 + \frac{C}{\alpha}    \norm{V_1}_{H^{\sfrac{3}{2}}}^2 \norm{ \delta V}_{L^2}^2, \\
\av{\ps{ B_2\pare{\delta V, V_2}}{\delta V}_{L^2}} & \leqslant \norm{ \nabla V_2}_{L^2}\norm{\delta V}_{L^4}^2, \\
& \leqslant \alpha \norm{\nabla \delta V}_{L^2}^2 + \frac{C}{\alpha}  \norm{\nabla V_2}_{L^2}^2\norm{\delta V}_{L^2}^2, \\
\av{\ps{ B_1\pare{V_1, \delta V}}{\delta V}_{L^2}}
 + \av{\ps{ B_1\pare{\delta V, V_2}}{\delta V}_{L^2}} & \leqslant \pare{\norm{V_1}_{L^2} + \norm{V_2}_{L^2}}\norm{\delta V}_{L^4}^2, \\
 & \leqslant \alpha \norm{\nabla\delta V}_{L^2}^2 + \frac{C}{\alpha} \pare{\norm{V_1}_{L^2}^2 + \norm{V_2}_{L^2}^2 } \norm{\delta V}_{L^2}^2, \\
 \av{\ps{L \ \delta V}{\delta V}_{L^2}} & \leqslant \alpha \norm{\nabla\delta V}_{L^2}^2 + \frac{C}{\alpha}\norm{\delta V}_{L^2}^2. 
\end{aligned}
\end{equation}

Whence considering the estimates \eqref{eq:ineq_unicity_-2} in \eqref{eq:ineq_unicity_-3}, selecting an $0<4 \alpha<\sfrac{c}{2} $ and applying a Gronwall inequality we can deduce the following rather crude bound, for any $ t>0 $
\begin{equation*}
\norm{\delta V\pare{t}}_{L^2}^2 + \frac{c}{2} \int_0^t \norm{\nabla \delta V\pare{\tau}}_{L^2}^2\d \tau \leqslant C \norm{\delta V_0}_{L^2}^2 \exp\set{C \int_0^t \phi\pare{\tau}\d \tau},
\end{equation*}
where
\begin{equation*}
\phi = \pare{1+ \norm{V_1}_{L^2}^2 + \norm{V_2}_{L^2}^2} \pare{1+ \norm{  V_1}_{H^{\sfrac{3}{2}}}^2 + \norm{  V_2}_{H^{\sfrac{3}{2}}}^2 }.
\end{equation*}

\noindent The function $ \phi \in L^1_{\loc}\pare{\bR_+} $ thanks to the results proved in Section \ref{sec:weak_solutions}, whence we conclude the proof of the first step.

 \subsubsection{Step 2: The $ \Lambda^{\sfrac{1}{2}} L^2 $ energy bound :}
 With a procedure which is now familiar we multiply \eqref{eq:Rosensweig2D_difference_compact_form} for $ \Lambda \delta V $ and integrate in space in order to deduce the energy inequality 
 \begin{multline}\label{eq:ineq_unicity_0}
 \frac{1}{2}\ \frac{\d}{\d t} \norm{\delta V}_{\sfrac{1}{2}}^2 + c\norm{\nabla\delta V}_{\sfrac{1}{2}}^2 \leqslant
 \av{\ps{ B_2\pare{V_1, \delta V}}{\delta V}_{\sfrac{1}{2}}}
 + \av{\ps{ B_2\pare{\delta V, V_2}}{\delta V}_{\sfrac{1}{2}}}\\
 + \av{\ps{ B_1\pare{V_1, \delta V}}{\delta V}_{\sfrac{1}{2}}}
 + \av{\ps{ B_1\pare{\delta V, V_2}}{\delta V}_{\sfrac{1}{2}}}
 +\av{\ps{L \ \delta V}{\delta V}_{\sfrac{1}{2}}}. 
\end{multline}  
Since as far as concerns energy estimates we can identify the bilinear form $ B_2 $ with the transport form up to a constant we will do so in order to simplify the notation of the proof. We consider at firs the term $ \av{ \ps{ B_2\pare{V_1, \delta V}}{\delta V}_{\sfrac{1}{2}}} $ which, as explained we identify with $ \av{ \ps{ {V_1\cdot \nabla \delta V}}{\delta V}_{\sfrac{1}{2}}} $. With a standard integration by parts we argue that
\begin{equation*}
\av{ \ps{ {V_1\cdot \nabla \delta V}}{\delta V}_{\sfrac{1}{2}}} \leqslant \av{ \ps{ {\nabla  V_1 \otimes  \delta V}}{\delta V}_{\sfrac{1}{2}}} + \av{ \ps{ { V_1 \otimes \delta V}}{\nabla \delta V}_{\sfrac{1}{2}}}.
\end{equation*}

\noindent We analyze at first the term $ \av{ \ps{ {\nabla \ V_1 \otimes  \delta V}}{\delta V}_{\sfrac{1}{2}}}  $, since the operator $ \Lambda^{\sfrac{1}{2}} $ is self-adjoint in $ L^2 $ we argue that
\begin{align*}
\av{ \ps{ {\nabla \ V_1 \otimes  \delta V}}{\delta V}_{\sfrac{1}{2}}} & = 
\av{ \int \nabla \ V_1 \otimes \delta V \ \Lambda \delta V \dx }, \\
&\leqslant \norm{ \nabla \ V_1}_{L^2} \norm{\delta V}_{L^4} \norm{\nabla \delta V}_{L^4}, 
\end{align*}
while since $ \Lambda^{\sfrac{1}{2}} L^2 $ embeds continuously in $ L^4 $ we deduce that
\begin{align}
\label{eq:ineq_unicity_0}
\av{ \ps{ {\nabla \ V_1 \otimes  \delta V}}{\delta V}_{\sfrac{1}{2}}} &\leqslant  \alpha \norm{\nabla\delta V}_{\sfrac{1}{2}}^2 + \frac{C}{\alpha}\norm{\nabla V_1}_{L^2}^2\norm{\delta V}_{\sfrac{1}{2}}^2. 
\end{align}

Using Lemma \ref{lem:product rule} instead
\begin{equation}
\label{eq:ineq_unicity_1}
\begin{aligned}
\av{ \ps{ { V_1 \otimes \delta V}}{\nabla \delta V}_{\sfrac{1}{2}}} & \lesssim \norm{V_1 \otimes \delta V}_{\sfrac{1}{2}} \norm{\nabla \delta V}_{\sfrac{1}{2}}, \\
& \lesssim \pare{\norm{V_1}_{L^\infty}\norm{\delta V}_{\sfrac{1}{2}} + \norm{V_1}_{\sfrac{1}{2}}\norm{\delta V}_{L^\infty} }\norm{\nabla \delta V}_{\sfrac{1}{2}}.
\end{aligned}
\end{equation}
Since $ H^{\sfrac{3}{2}}\hra L^\infty $ we deduce that
\begin{equation}
\label{eq:ineq_unicity_2}
\norm{V_1}_{L^\infty} \lesssim \norm{V_1}_{H^{\sfrac{3}{2}}}, 
\end{equation}
while using \eqref{eq:interpolation_inequality} and the inequality $ \norm{v}_{\sfrac{1}{2}}\leqslant C \norm{v}_{L^2}^{\sfrac{1}{2}} \norm{ \nabla v}_{L^2}^{\sfrac{1}{2}} $ we deduce 
\begin{equation}
\label{eq:ineq_unicity_3}
\norm{V_1}_{\sfrac{1}{2}}\norm{\delta V}_{L^\infty} \leqslant C \norm{V_1}_{L^2}^{\sfrac{1}{2}} \norm{ \nabla V_1}_{L^2}^{\sfrac{1}{2}} \norm{\delta V}_{\sfrac{1}{2}}^{\sfrac{1}{2}}\norm{\nabla\delta V}_{\sfrac{1}{2}}^{\sfrac{1}{2}}. 
\end{equation}
Whence with \eqref{eq:ineq_unicity_2} and \eqref{eq:ineq_unicity_3} the inequality \eqref{eq:ineq_unicity_1} becomes
\begin{equation}\label{eq:ineq_unicity_init}
\begin{aligned}
\av{ \ps{ { V_1 \otimes \delta V}}{\nabla \delta V}_{\sfrac{1}{2}}} & \leqslant C \pare{\norm{V_1}_{H^{\sfrac{3}{2}}}\norm{\delta V}_{\sfrac{1}{2}} + \norm{V_1}_{L^2}^{\sfrac{1}{2}} \norm{ \nabla V_1}_{L^2}^{\sfrac{1}{2}} \norm{\delta V}_{\sfrac{1}{2}}^{\sfrac{1}{2}}\norm{\nabla\delta V}_{\sfrac{1}{2}}^{\sfrac{1}{2}} }\norm{\nabla \delta V}_{\sfrac{1}{2}}, \\
&\leqslant \alpha \norm{\nabla \delta V}_{\sfrac{1}{2}}^2 +\frac{C}{\alpha} \pare{\norm{V_1}_{H^{\sfrac{3}{2}}}^2 + \norm{V_1}_{L^2}^{2} \norm{ \nabla V_1}_{L^2}^{2}} \norm{\delta V}_{\sfrac{1}{2}}^{2}. 
\end{aligned}
\end{equation}

The term $ {\ps{ B_2\pare{\delta V, V_2}}{\delta V}_{\sfrac{1}{2}}} $, on which we perform the identification  $ B_2\pare{\delta V, V_2}\sim \delta V\cdot \nabla V_2 $, can be bounded as
\begin{equation}
\begin{aligned}
\av{\ps{\delta V\cdot \nabla V_2}{\delta V}_{\sfrac{1}{2}}} & \leqslant C \norm{\delta V\cdot \nabla V_2}_{-\sfrac{1}{2}} \norm{\nabla\delta V}_{\sfrac{1}{2}}, \\
& \leqslant C \norm{\nabla V_2}_{L^2} \norm{\delta V}_{\sfrac{1}{2}}\norm{\nabla\delta V}_{\sfrac{1}{2}}, \\
& \leqslant \alpha \norm{\nabla\delta V}_{\sfrac{1}{2}}^2 +\frac{C}{\alpha} \norm{\nabla V_2}_{L^2}^2 \norm{\delta V}_{\sfrac{1}{2}}^2. 
\end{aligned}
\end{equation}

Next we consider the bilinear interactions generated by $ B_1 $ (which we identify as $ B_1\pare{A, B}\sim A\otimes B $ for the energy estimates). The following estimates are immediate
\begin{equation}
\begin{aligned}
\av{\ps{V_1\otimes \delta V}{\delta V}_{\sfrac{1}{2}}} & \leqslant \norm{V_1\otimes \delta V}_{- \sfrac{1}{2}} \norm{\nabla\delta V}_{\sfrac{1}{2}} , \\
& \leqslant \alpha \norm{\nabla\delta V}_{\sfrac{1}{2}}^2 + \frac{C}{\alpha} \norm{V_1}_{L^2}^2 \norm{\delta V}_{\sfrac{1}{2}}^2, \\
\av{\ps{\delta V\otimes \delta V_2}{\delta V}_{\sfrac{1}{2}}} & \leqslant \alpha \norm{\nabla\delta V}_{\sfrac{1}{2}}^2 + \frac{C}{\alpha} \norm{V_1}_{L^2}^2 \norm{\delta V}_{\sfrac{1}{2}}^2. 
\end{aligned}
\end{equation}

Lastly we assert that we can bound
\begin{equation}\label{eq:ineq_unicity_fin}
\av{\ps{L\ \delta V}{\delta V}_{\sfrac{1}{2}}} \leqslant \alpha \norm{\nabla \delta V}_{\sfrac{1}{2}}^2 + \frac{ C}{\alpha} \norm{ \delta V}_{\sfrac{1}{2}}^2. 
\end{equation}

At this point considering the bounds \eqref{eq:ineq_unicity_0},  \eqref{eq:ineq_unicity_init}--\eqref{eq:ineq_unicity_fin} in \eqref{eq:ineq_unicity_0} and selecting $ \alpha > 0 $ sufficiently small so that $ c-5\alpha \geqslant \hat{c} > 0 $ we can deduce the inequality
\begin{equation*}
\frac{1}{2} \ \frac{\d}{\d t} \ \norm{ \delta V\pare{t}}_{\sfrac{1}{2}}^2 + \hat{c} \norm{\nabla \delta  V\pare{t}}_{\sfrac{1}{2}}^2 \leqslant C_{\alpha} \ h\pare{t} \norm{\delta V\pare{t}}_{\sfrac{1}{2}}^2, 
\end{equation*}
where
\begin{equation*}
h\pare{t} =  \pare{1 +\norm{V_1\pare{t}}_{L^2}^2 + \norm{V_2\pare{t}}_{L^2}^2 } \pare{1 + \norm{V_1\pare{t}}_{H^{\sfrac{3}{2}}}^2 + \norm{V_2\pare{t}}_{H^{\sfrac{3}{2}}}^2 }, 
\end{equation*}
whence applying a Gronwall inequality we deduce
\begin{equation*}
\norm{\delta V \pare{t}}_{\sfrac{1}{2}}^2 + 2\hat{c} \int_0^t \norm{\nabla\delta  V \pare{t'}}_{\sfrac{1}{2}}^2 \d t' \leqslant C_\alpha \norm{\delta V_0}_{\sfrac{1}{2}}^2 \exp \set{\int_0^t h \pare{t'}\d t' }.  
\end{equation*}

Thanks to the results proved in Section \ref{sec:weak_solutions} we know hence that $ h\in L^1_{\loc}\pare{\bR_+} $, whence we deduce the uniqueness of the weak solutions in the energy space $ L^\infty_{\loc}\pare{ \bR_+; H^{\sfrac{1}{2}}}\cap L^2_{\loc} \pare{ \bR_+;  H^{\sfrac{3}{2}}} $, concluding. \hfill $ \Box $

\subsection{Proof of Lemma \ref{lem:higher_energy_est_transport_term}}\label{sec:pf_lem_higher_energy_est_transport_term}
Let us recall that the following Leibniz rule applies
\begin{equation*}
\partial^\alpha\pare{a\cdot \nabla b} = \sum_{{\beta}+ {\gamma} = {\alpha}} \kappa_{ \beta, \gamma}\ \partial^\beta a \cdot \nabla \partial^\gamma b, 
\end{equation*}
where the $ \kappa_{\beta, \gamma} $ are positive and finite and their explicit value is irrelevant in our context. \\
Whence
\begin{equation}\label{eq:Leibniz_rule}
\av{ \int \pa \pare{a\cdot \nabla b} \ \pa c \ \dx} \leqslant 
\sum_{{\beta}+ {\gamma} = {\alpha}} \kappa_{ \beta, \gamma}\ 
\av{ \int \partial^\beta a \cdot \nabla \partial^\gamma b \ \cdot \ \pa c \ \dx}. 
\end{equation}

We divide now the right hand side of the above inequality in three cases
\begin{itemize}
\item If $ \av{\beta}, \av{\gamma} \geqslant 1 $ we can bound the right hand side of \eqref{eq:Leibniz_rule}, restricted on such set,  as 
\begin{multline*}
\sum_{\substack{{\beta}+ {\gamma} = {\alpha}\\ \av{\beta}, \av{\gamma}\geqslant 1}} \kappa_{\alpha, \beta}\ 
\av{ \int \partial^\beta a \cdot \nabla \partial^\gamma b \ \cdot \ \pa c \ \dx}\\
\begin{aligned}
&\leqslant
 \sum_{\substack{{\beta}+ {\gamma} = {\alpha}\\ \av{\beta}, \av{\gamma}\geqslant 1}} \kappa_{ \beta, \gamma} \norm{\partial^\beta a}_{L^2}^{\sfrac{1}{2}}\norm{\nabla\partial^\beta a}_{L^2}^{\sfrac{1}{2}}
\norm{\partial^\alpha c}_{L^2}^{\sfrac{1}{2}}\norm{\nabla\partial^\alpha c}_{L^2}^{\sfrac{1}{2}}\norm{\nabla\partial^\gamma b}_{L^2}, \\
& \leqslant \sum_{\substack{{\beta}+ {\gamma} = {\alpha}\\ \av{\beta}, \av{\gamma}\geqslant 1}} \pare{
\kappa_{\beta, \gamma}^4 \tilde{\epsilon} \norm{\nabla\partial^\alpha c}_{L^2}^{2}
+ \frac{C}{\tilde{\epsilon}}
\norm{\partial^\beta a}_{L^2}^{2}\norm{\nabla\partial^\beta a}_{L^2}^{2}
\norm{\partial^\alpha c}_{L^2}^{2}
+ \frac{C}{\tilde{\epsilon}}
\norm{\nabla\partial^\gamma b}_{L^2}^2
}, \\
&
\begin{multlined}
\leqslant \pare{\sum_{\substack{{\beta}+ {\gamma} = {\alpha}\\ \av{\beta}, \av{\gamma}\geqslant 1}}\kappa_{\beta, \gamma}^4} \tilde{\epsilon}\norm{\nabla\partial^\alpha c}_{L^2}^{2}  
+\frac{C}{\tilde{\epsilon}}
\pare{\sum_{\substack{1\leqslant \av{\beta}< k}}\norm{\partial^\beta a}_{L^2}^{2}\norm{\nabla\partial^\beta a}_{L^2}^{2}
}\norm{ \partial^\alpha c}_{L^2}^{2}  
\\
+\frac{C}{\tilde{\epsilon}}
\pare{ \sum_{\substack{1\leqslant \av{\gamma}< k}} \norm{\nabla\partial^\gamma b}_{L^2}^2}. 
\end{multlined}
\end{aligned}
\end{multline*}

Accordingly to the hypothesis \eqref{eq:technical_bound} we can assert that there exist a constant $ c_k $ depending only on the length of the multi-index $  \alpha $ such that
\begin{align*}
\sum_{\substack{1\leqslant \av{\beta}< k}}\norm{\partial^\beta a}_{L^2}^{2}\norm{\nabla\partial^\beta a}_{L^2}^{2} & \leqslant c_k \sum_{\ell=1}^{k-1}F_\ell^2 G_\ell^2 , \\
\sum_{\substack{1\leqslant \av{\gamma}< k}} \norm{\nabla\partial^\gamma b}_{L^2}^2 & \leqslant c_k \sum_{\ell=1}^{k-1} G_\ell^2. 
\end{align*}

Whence since the values $ \kappa_{\beta, \gamma} $ are finite, for any $ \epsilon > 0 $ we can select a $  \tilde{\epsilon} $ sufficiently small, which again depends on $ k $ only,  so that
\begin{equation*}
\pare{\sum_{\substack{{\beta}+ {\gamma} = {\alpha}\\ \av{\beta}, \av{\gamma}\geqslant 1}}\kappa_{\beta, \gamma}^4} \tilde{\epsilon} < \frac{\epsilon}{8},
\end{equation*}

We hence proved that
\begin{equation} \label{eq:tecnical_bound_1}
\sum_{\substack{\av{\beta}+ \av{\gamma} = \av{\alpha}\\ \av{\beta}, \av{\gamma}\geqslant 1}} \kappa_{\alpha, \beta}\ 
\av{ \int \partial^\beta a \cdot \nabla \partial^\gamma b \ \cdot \ \pa c \ \dx} 
\leqslant  \frac{\epsilon}{8} \ \norm{\nabla\partial^\alpha c}_{L^2}^{2}  + \frac{C c_k}{\epsilon} \pare{\sum_{\ell=1}^{k-1}F_\ell^2 G_\ell^2} \norm{ \partial^\alpha c}_{L^2}^{2}
+ \frac{C c_k}{\epsilon} \sum_{\ell=1}^{k-1} G_\ell^2. 
\end{equation}

\item If $ \beta = 0 $ and $ \gamma = \alpha $ the right-hand side of \eqref{eq:Leibniz_rule} restricted in such set can be bounded as
\begin{equation}
\label{eq:bound_0alpha_1}
\begin{aligned}
\kappa_{0, \alpha} \av{\int a \cdot \nabla \pa b \ \cdot \ \pa c \ \dx} & \leqslant \kappa_{0, \alpha} \norm{a}_{L^2}^{\sfrac{1}{2}} \norm{\nabla a}_{L^2}^{\sfrac{1}{2}} \norm{\nabla \pa b}_{L^2}  \norm{\pa c}_{L^2}^{\sfrac{1}{2}} \norm{\nabla \pa c}_{L^2}^{\sfrac{1}{2}} , \\
& \leqslant \pare{ \kappa_{0, \alpha}^4 \tilde{\epsilon} \norm{\nabla \pa c}_{L^2}^2 + \tilde{\epsilon}\norm{\nabla \pa b}_{L^2}^2} + \frac{C}{\tilde{\epsilon}} \norm{a}_{L^2}^{2} \norm{\nabla a}_{L^2}^{2} \norm{\pa c}_{L^2}^{2}. 
\end{aligned}
\end{equation}

Thanks to the hypothesis \eqref{eq:technical_bound} we assert that
\begin{equation*}
\norm{a}_{L^2}^{2} \norm{\nabla a}_{L^2}^{2}\leqslant F_0^2 G_0^2, 
\end{equation*}
and moreover we can choose $ \tilde{\epsilon} $ sufficiently small so that $ 2\max \set{\kappa_{0, \alpha}^4, 1}\tilde{\epsilon}<\dfrac{\epsilon}{8} $, whence \eqref{eq:bound_0alpha_1} becomes
\begin{equation}\label{eq:tecnical_bound_2}
\kappa_{0, \alpha} \av{\int a \cdot \nabla \pa b \ \cdot \ \pa c \ \dx} \leqslant \frac{\epsilon}{8} \pare{ \norm{\nabla \pa c}_{L^2}^2 + \norm{\nabla \pa b}_{L^2}} + \frac{C}{{\epsilon}} F_0^2 G_0^2 \norm{\pa c}_{L^2}^{2}.
\end{equation}

\item If $ \beta =\alpha $ and $ \gamma =0 $ the right hand side of \eqref{eq:Leibniz_rule} becomes
\begin{equation*}
\begin{aligned}
\kappa_{ \alpha , 0 } \av{\int \pa a \cdot \nabla  b \ \cdot \ \pa c \ \dx} & \leqslant \norm{\pa a}_{L^2}^{\sfrac{1}{2}} \norm{\nabla \pa a}_{L^2}^{\sfrac{1}{2}} \norm{\nabla b}_{L^2} \norm{\pa c}_{L^2}^{\sfrac{1}{2}} \norm{\nabla \pa c}_{L^2}^{\sfrac{1}{2}}, \\
& \leqslant
 \tilde{\epsilon} \pare{ \norm{\nabla \pa a}_{L^2}^{2} + \norm{\nabla \pa c}_{L^2}^{2} } + \frac{C}{\tilde{\epsilon}}\norm{\nabla b}_{L^2}^2
 \pare{ \norm{ \pa a}_{L^2}^{2} + \norm{ \pa c}_{L^2}^{2} }, 
\end{aligned}
\end{equation*}
Whence considering the hypothesis \eqref{eq:technical_bound} and selecting a $ \tilde{\epsilon} $ sufficiently small we deduce
\begin{equation}\label{eq:tecnical_bound_3}
\kappa_{ \alpha , 0 } \av{\int \pa a \cdot \nabla  b \ \cdot \ \pa c \ \dx}
\leqslant 
\frac{\epsilon}{8} \pare{ \norm{\nabla \pa a}_{L^2}^{2} + \norm{\nabla \pa c}_{L^2}^{2} } + \frac{C}{{\epsilon}} \ G_0^2
 \pare{ \norm{ \pa a}_{L^2}^{2} + \norm{ \pa c}_{L^2}^{2} }. 
\end{equation}

\end{itemize}

Summing hence \eqref{eq:tecnical_bound_1}, \eqref{eq:tecnical_bound_2} and \eqref{eq:tecnical_bound_3}, and setting $ C_k \gg C c_k $ we deduce the inequality \eqref{eq:bound_Leibniz_rule}. 
\hfill $ \Box $

\subsection{Proof of Lemma \ref{lem:higher_energy_est_bilinear_term}}
\label{sec:pf_lem_higher_energy_est_bilinear_term}
Again as in the proof of Lemma \ref{lem:higher_energy_est_transport_term} we exploit Leibniz formula
\begin{equation} \label{eq:bilinear_bound_high_energy_1}
\av{\int \pa \pare{a \  b} \ \pa c \ \dx}  = 
\sum_{{\beta} +\gamma = \alpha} \kappa_{\beta, \gamma} \av{\int \partial^\beta a \ \partial^\gamma b \ \pa c \ \dx}, 
\end{equation}
and we divide the proof in three cases:
\begin{itemize}
\item Suppose $ \av{\beta}, \av{\gamma} \geqslant 1 $, whence the right hand side of \eqref{eq:bilinear_bound_high_energy_1} restricted on such set can be bounded as
\begin{align*}
\sum_{{\beta} +\gamma = \alpha} \kappa_{\beta, \gamma} \av{\int \partial^\beta a \ \partial^\gamma b \ \pa c \ \dx} & \leqslant  \sum_{{\beta} +\gamma = \alpha} \kappa_{\beta, \gamma} \norm{\partial^\beta a}_{L^4} \norm{\partial^\gamma b}_{L^4} \norm{\pa c}_{L^2}, \\
& \leqslant \sum_{{\beta} +\gamma = \alpha} \kappa_{\beta, \gamma} \norm{\partial^\beta a}_{L^2}^{\sfrac{1}{2}} \norm{\nabla \partial^\beta a}_{L^2}^{\sfrac{1}{2}} \norm{\partial^\gamma b}_{L^2}^{\sfrac{1}{2}}\norm{\nabla\partial^\gamma b}_{L^2}^{\sfrac{1}{2}} \norm{\pa c}_{L^2},
\end{align*}
while using the hypothesis \ref{eq:technical_bound} we deduce
\begin{equation}\label{eq:bilinear_bound_high_energy_2}
\sum_{{\beta} +\gamma = \alpha} \kappa_{\beta, \gamma} \av{\int \partial^\beta a \ \partial^\gamma b \ \pa c \ \dx} 
\leqslant
C_k 
\pare{ 
\sum_{\ell = 1}^{k-1} G_{\ell-1}^{\sfrac{1}{2}} G_{\ell}^{\sfrac{1}{2}} G_{k-\ell-1}^{\sfrac{1}{2}} G_{k-\ell}^{\sfrac{1}{2}}
} \norm{\pa c}_{L^2} . 
\end{equation}

\item If $ \beta =\alpha $ and $ \gamma=0 $ then
\begin{align*}
\av{\int \pa a \ b \ \pa c \ \dx} & \leqslant \norm{\pa a}_{L^4} \norm{b}_{L^2} \norm{\pa c}_{L^4},  \\
& \leqslant \frac{C}{\tilde{\epsilon}} \  \norm{\pa a}_{L^2} \norm{\pa c}_{L^2} \norm{b}_{L^2}^2+ \tilde{\epsilon} \pare{\norm{\nabla \pa a}_{L^2}^{2} + \norm{\nabla \pa c}_{L^2}^{2}} ,
\end{align*}
and since
\begin{equation*}
\norm{\pa a}_{L^2} \norm{\pa c}_{L^2} \leqslant \norm{\nabla a}_{H^{k-1}} \norm{\nabla c}_{H^{k-1}} \leqslant G_{k-1}^2, 
\end{equation*}
we deduce
\begin{equation}\label{eq:bilinear_bound_high_energy_3}
\av{\int \pa a \ b \ \pa c \ \dx} \leqslant
\tilde{\epsilon} \pare{\norm{\nabla \pa a}_{L^2}^{2} + \norm{\nabla \pa c}_{L^2}^{2}}
+
\frac{C}{\tilde{\epsilon}} \  G_{k-1}^2 \norm{b}_{L^2}^2. 
\end{equation}

\item If $ \beta =0 $ and $ \gamma=\alpha $ then the bound we look for is the following
\begin{equation}\label{eq:bilinear_bound_high_energy_4}
\av{\int  a \ \pa b \ \pa c \ \dx} \leqslant
\tilde{\epsilon} \pare{\norm{\nabla \pa b}_{L^2}^{2} + \norm{\nabla \pa c}_{L^2}^{2}}
+
\frac{C}{\tilde{\epsilon}} \  G_{k-1}^2 \norm{a}_{L^2}^2, 
\end{equation}
and whose proof is identical but symmetric of the one performed above, and hence omitted. 
\end{itemize}

Inserting the results given in \eqref{eq:bilinear_bound_high_energy_2}, \eqref{eq:bilinear_bound_high_energy_3} and \eqref{eq:bilinear_bound_high_energy_4} in \eqref{eq:bilinear_bound_high_energy_1}, selecting $ \tilde{\epsilon} $ sufficiently small and $ C $ sufficiently large we deduce \eqref{eq:bilinear_bound_high_energy}. 
\hfill $ \Box $

\footnotesize{
\providecommand{\bysame}{\leavevmode\hbox to3em{\hrulefill}\thinspace}
\providecommand{\MR}{\relax\ifhmode\unskip\space\fi MR }
\providecommand{\MRhref}[2]{%
  \href{http://www.ams.org/mathscinet-getitem?mr=#1}{#2}
}
\providecommand{\href}[2]{#2}

}

\end{document}